\documentclass[11pt, final]{article}
\usepackage{a4}

\usepackage{amsmath}
\usepackage{amssymb}
\usepackage{amsxtra}
\usepackage{amsthm}
\usepackage{bbm}			% \mathbb{1} funktioniert nicht, \mathbbm{1} schon
\usepackage{cite}

\usepackage[left=3.5cm,right=3.5cm,top=3cm,bottom=3cm]{geometry}

\usepackage[plainpages=false,pdfpagelabels]{hyperref} % erzeugt die Links in der PDF

\usepackage{showkeys}
\usepackage[latin1]{inputenc}       % F?r Umlaut und ? beim Schreiben
\usepackage[T1]{fontenc}        % f?r die Schriftart
\usepackage{lmodern}            % braucht man f?r die Skalierung der Schriften
\usepackage{graphicx}
\usepackage{subfigure}		% für mehrere Bilder in einer Abbildung
\usepackage{booktabs}		% für toprule, midrule und bottomrule in tabellen
\usepackage{multirow}
\usepackage{enumitem}		% für besondere Labels bei Aufzählungen
\usepackage{listings} 	% für code-Umgebungen
\usepackage{empheq}			% für box um math-umgebung
\usepackage{mathtools} 	% für mathclap{} - bessere Ausrichtung von Formeln an Ungleichheitszeichen usw. (bei z.B. Text über dem Zeichen)

\usepackage{colortbl}
\definecolor{dunkelgrau}{rgb}{0.8,0.8,0.8}
\definecolor{hellgrau}{rgb}{0.9,0.9,0.9}

\newcommand{\R}{\ensuremath{\mathbb{R}}}

\newcommand{\oR}{\ensuremath{\overline{\mathbb{R}}}}
\renewcommand{\O}{\ensuremath{\mathcal{O}}}

\renewcommand{\>}{\right\rangle}
\newcommand{\<}{\left\langle}

\newcommand{\Prox}{\ensuremath{\text{Prox}}}

\newcommand{\Z}{\ensuremath{\mathcal{Z}}}

\newcommand{\h}{\ensuremath{\mathcal{H}}}
\newcommand{\hk}{\ensuremath{\mathcal{K}}}

\newcommand{\D}{\ensuremath{\mathcal{D}}}

\theoremstyle{plain}
\newtheorem{theorem}{Theorem}

\theoremstyle{definition}

\DeclareMathOperator*\dom{dom}%
\DeclareMathOperator*\argmin{arg\,min}%

\title{A variable smoothing algorithm for solving convex optimization problems}

\author{Radu Ioan Bo\c t
\thanks {Faculty of Mathematics, Chemnitz University of Technology, D-09107 Chemnitz, Germany, e-mail: radu.bot@mathematik.tu-chemnitz.de. Research partially supported by DFG (German Research Foundation), project BO 2516/4-1.}
\and Christopher Hendrich
\thanks{Faculty of Mathematics, Chemnitz University of Technology, D-09107 Chemnitz, Germany, e-mail: christopher.hendrich@mathematik.tu-chemnitz.de.}
}
\date{\today}

\begin{document}
\maketitle

{\bf Abstract.} In this article we propose a method for solving unconstrained optimization problems with convex and Lipschitz continuous objective functions. By making use of the Moreau envelopes of the functions occurring in the objective, we smooth the latter to a convex and differentiable function with Lipschitz continuous gradient by using both variable and constant smoothing parameters. The resulting problem is solved via an accelerated first-order method and this allows us to recover approximately the optimal solutions to the initial optimization problem with a rate of convergence of order $\O(\tfrac{\ln k}{k})$ for variable smoothing and of order $\O(\tfrac{1}{k})$ for constant smoothing. Some numerical experiments employing the variable smoothing method in image processing and in supervised learning classification are also presented.



\section{Introduction}\label{sectionIntro}

In this paper we introduce and investigate the convergence properties of an efficient algorithm for solving nondifferentiable optimization problems of type
\begin{equation}
\label{eq:primal-original}
\inf_{x \in \h}{\left\{f(x)+g(Kx)\right\}},
\end{equation}
where $\h$ and $\hk$ are real Hilbert spaces, $f : \h \rightarrow \R$ and $g:\hk \rightarrow \R$ are convex and Lipschitz continuous functions and the operator $K:\h \rightarrow \hk$ is linear and continuous. By replacing the functions $f$ and $g$ through their Moreau envelopes, approach which can be seen as part of the family of smoothing techniques introduced in \cite{NesterovExcessiveGap05,NesterovSmoothMin05,NesterovSmoothing05}, we approximate \eqref{eq:primal-original} by a convex optimization problem with a differentiable objective function with Lipschitz continuous gradient. This smoothing approach can be seen as the counterpart of the so-called double smoothing method investigated in \cite{BotHendrich12,BotHendrich12b,DevGliNes12}, which assumes the smoothing of the Fenchel-dual problem to \eqref{eq:primal-original} to an optimization problem with a strongly convex and differentiable objective function with Lipschitz continuous gradient. There, the smoothed dual problem is solved  via an appropriate fast gradient method (cf. \cite{Nesterov04}) and a primal optimal solution is reconstructed with a given level of accuracy. In contrast to that approach, which asks for the boundedness of the effective domains of $f$ and $g$, determinant is here the boundedness of the effective domains of the conjugate functions $f^*$ and $g^*$, which is automatically guaranteed by the Lipschitz continuity of $f$ and $g$, respectively. For solving the resulting smoothed problem we propose an extension of the accelerated gradient method of Nesterov (cf. \cite{Nesterov83}) for convex optimization problems involving variable smoothing parameters which are updated in each iteration. This scheme yields for the minimization of the objective of the initial problem a rate of convergence of order $\O(\tfrac{\ln k}{k})$, while, in the particular case when the smoothing parameters are constant, the order of the  rate of convergence becomes $\O(\tfrac{1}{k})$. Nonetheless, using variable smoothing parameters has an important advantage, although the theoretical rate of convergence is not as good as when these are constant. In the first case the approach generates a sequence of iterates $(x_k)_{k \geq 1}$ such that $(f(x_k)+g(Kx_k))_{k \geq 1}$ converges to the optimal objective value of \eqref{eq:primal-original}. In the case of constant smoothing variables the approach provides a sequence of iterates which solves the problem \eqref{eq:primal-original} with an apriori given accuracy, however, the sequence  $(f(x_k)+g(Kx_k))_{k \geq 1}$ may not converge to the optimal objective value of the problem to be solved.

In addition, we show, on the one hand, that the two approaches can be designed and keep the same convergence behavior also in the case when $f$ is differentiable with Lipschitz continuous gradient and, on the other hand, that they can be employed also for solving the extended version of \eqref{eq:primal-original}
 \begin{equation}
\label{eq:primal-sum}
\inf_{x \in \h}{\left\{f(x)+\sum_{i=1}^m{g_i(K_ix)}\right\}},
\end{equation}
where $\hk_i$ are real Hilbert spaces, $g_i : \hk_i \rightarrow \R$ are convex and Lipschitz continuous functions and $K_i: \h \rightarrow \hk_i$, $i=1,\ldots,m$, are linear continuous operators.

The structure of this paper is as follows. In Section \ref{sectionPrelimFormulation} we recall some elements of convex analysis and establish the working framework. Section \ref{sectionPS} is mainly devoted to the description of the iterative methods for solving \eqref{eq:primal-original} and of their convergence properties for both variable and constant smoothing and to the presentation of some of their variants. In Section \ref{sectionExample} numerical experiments employing the variable smoothing method in image processing and in supervised vector machines classification are presented.

\section{Preliminaries of convex analysis and problem formulation}\label{sectionPrelimFormulation}

In the following we are considering the real Hilbert spaces $\h$ and $\hk$ endowed with the inner product $\left\langle \cdot ,\cdot \right\rangle$ and associated norm $\left\| \cdot \right\| = \sqrt{\left\langle \cdot, \cdot \right\rangle}$. By $B_{\h} \subseteq \h$ and $\R_{++}$ we denote the \textit{closed unit ball} of $\h$ and the set of strictly positive real numbers, respectively. The \textit{indicator function} of the set $C \subseteq \h$ is the function $\delta_C : \h \rightarrow  \oR := \R \cup \left\{ \pm \infty \right\}$ defined by $\delta_C(x) = 0$ for $x \in C$ and $\delta_C(x) = +\infty$, otherwise. For a function $f: \h \rightarrow \oR$ we denote by $\dom f := \left\{ x \in \h : f(x) < +\infty \right\}$ its \textit{effective domain}. We call $f$ \textit{proper} if $\dom f \neq \emptyset$ and $f(x)>-\infty$ for all $x \in \h$. The \textit{conjugate function} of $f$ is $f^*:\h \rightarrow \overline{\mathbb{R}}$, $f^*(p)=\sup{\left\{ \left\langle p,x \right\rangle -f(x) : x\in\h \right\}}$ for all $p \in \h$. The \textit{biconjugate function} of $f$ is $f^{**} : \h \rightarrow \oR$, $f^{**}(x) = \sup{\left\{ \left\langle x,p \right\rangle -f^*(p) : p\in\h \right\}}$ and, when $f$ is proper, convex and lower semicontinuous, according to the Fenchel-Moreau Theorem, one has $f=f^{**}$. The \textit{(convex) subdifferential} of the function $f$ at $x \in \h$ is the set $\partial f(x) = \{p \in \h : f(y) - f(x) \geq \left\langle p,y-x \right\rangle \ \forall y \in \h\}$, if $f(x) \in \R$, and is taken to be the empty set, otherwise. For a linear operator $K: \h \rightarrow \hk$, the operator $K^*: \hk \rightarrow \h$ is the \textit{adjoint operator} of $K$ and is defined by $\left\langle K^*y,x  \right\rangle = \left\langle y,Kx  \right\rangle$ for all $x \in \h$ and all $y \in \hk$.

Having two functions $f,\,g : \h \rightarrow \overline{\mathbb{R}}$, their \textit{infimal convolution} is defined by $f \Box g : \h \rightarrow \oR$, $(f \Box g) (x) = \inf_{y \in \h}\left\{ f(y) + g(x-y) \right\} $ for all $x \in \h$. When $f,g : \h \rightarrow \overline \R$ are proper and convex, then
\begin{equation}\label{exinfconv}
(f+g)^* = f^* \Box g^*
\end{equation}
provided that $f$ (or $g$) is continuous at a point belonging to $\dom f \cap \dom g$. For other qualification conditions guaranteeing \eqref{exinfconv} we refer the reader to \cite{Bot10}.

The \textit{Moreau envelope} of parameter $\gamma \in \R_{++}$ of a proper, convex and lower semicontinuous function $f : \h \rightarrow \oR$  is the function $^{\gamma}f : \h \rightarrow \R$, defined as
$$^{\gamma}f(x) := f \Box \left(\frac{1}{2\gamma}\left\| \cdot \right\|^2\right)(x) = \inf_{y \in \h} \left \{f(y) + \frac{1}{2\gamma}\|x-y\|^2 \right\} \ \forall x \in \h.$$
For every $x \in \h$ we denote by $\Prox_{\gamma f}(x)$ the \textit{proximal point} of parameter $\gamma$ of $f$ at $x$, namely, the unique optimal solution of the optimization problem
\begin{equation}\label{prox-def}
\inf_{y\in \h}\left \{f(y)+\frac{1}{2\gamma}\|y-x\|^2\right\}.
\end{equation}
Notice that $\Prox_{\gamma f} :\h \rightarrow \h$ is single-valued and firmly nonexpansive (cf. \cite[Proposition 12.27]{BauschkeCombettes11}), i.e.,
\begin{equation}\label{firmnon}
\|\Prox_{\gamma f}(x) - \Prox_{\gamma f}(y)\|^2 + \|(x-\Prox_{\gamma f}(x)) - (y-\Prox_{\gamma f}(y)) \|^2\leq \|x-y\|^2 \ \forall x,y \in \h,
\end{equation}
thus $1$-Lipschitz continuous, i.e., Lipschitz continuous with Lipschitz constant equal to $1$. We also have (cf. \cite[Theorem 14.3]{BauschkeCombettes11})
\begin{equation}\label{prox-conj}
^{\gamma}f(x) + {}^{\tfrac{1}{\gamma}}f^*(\tfrac{x}{\gamma}) = \frac{\|x\|^2}{2\gamma}  \ \forall x \in \h
\end{equation}
and the extended \textit{Moreau's decomposition formula}
\begin{equation}\label{prox-f-star}
\Prox_{\gamma f}(x) + \gamma\Prox_{\tfrac{1}{\gamma}f^*}\left(\tfrac{x}{\gamma}\right)=x \ \forall x \in \h.
\end{equation}
The function $^{\gamma}f$ is (Fr\'echet) differentiable on $\h$ and its gradient $\nabla(^{\gamma}f) : \h \rightarrow \h$ fulfills  (cf. \cite[Proposition 12.29]{BauschkeCombettes11})
\begin{equation}\label{grad}
\nabla(^{\gamma}f)(x) = \tfrac{1}{\gamma}(x - \Prox_{\gamma f}(x)) \ \forall x \in \h,
\end{equation}
being in the light of \eqref{firmnon} $\tfrac{1}{\gamma}$-Lipschitz continuous. For a nonempty, convex and closed set $C \subseteq \h$ and $\gamma \in \R_{++}$ we have that $\Prox_{\gamma \delta_C} = \mathcal{P}_C$, where
$\mathcal{P}_C : \h \rightarrow C$, $\mathcal{P}_C(x) = \argmin_{z\in C}\left\| x-z \right\|$, denotes the \textit{projection operator} on $C$.

When $f:\h \rightarrow \R$ is convex and differentiable having an $L_{\nabla f}$-Lipschitz continuous gradient, then for all $x,y\in \h$ it holds (see, for instance, \cite{BauschkeCombettes11, Nesterov04, Nesterov83})
\begin{equation}\label{ineq-nest}
f(y) \leq f(x) + \<\nabla f(x),y-x \> + \frac{L_{\nabla f}}{2} \left\| y-x \right\|^2.
\end{equation}

The optimization problem that we investigate in this paper is
\begin{equation*}
\hspace{-1.8cm}(P) \quad \quad \inf_{x \in \h}{\left\{f(x)+g(Kx)\right\}},
\end{equation*}
where $K:\h \rightarrow \hk$ is a linear continuous operator and  $f:\h \rightarrow \R$ and $g:\hk \rightarrow \R$ are convex and $L_f$-Lipschitz continuous and $L_g$-Lipschitz continuous functions, respectively. According to \cite[Proposition 4.4.6]{BorVan10} we have that
\begin{equation}\label{dom-bound}
\dom f^* \subseteq L_fB_{\h} \ \mbox{and} \ \dom g^* \subseteq L_gB_{\hk}.
\end{equation}

\section{The algorithm and its variants}\label{sectionPS}
\subsection{The smoothing of the problem $(P)$}\label{subsectionPrimalSmooth}

The algorithms we would like to introduce and analyze from the point of view of their convergence properties assume in a first instance an appropriate smoothing of the problem $(P)$ which we are going to describe in the following.

For $\rho \in \R_{++}$ we smooth $f$ via its Moreau envelope of parameter $\rho$, $^{\rho}f : \h \rightarrow \R$, $^{\rho}f(x) = \left(f\Box\frac{1}{2\rho}\left\| \cdot \right\|^2\right)(x)$ for every $x \in \h$. According to the Fenchel-Moreau Theorem and due to \eqref{exinfconv}, one has for $x \in \h$
\begin{align*}
^{\rho}f(x) = \left(f^{**}\Box\frac{1}{2\rho}\left\| \cdot \right\|^2\right)(x)  = \left(f^* + \frac{\rho}{2}\left\| \cdot \right\|^2 \right)^*(x) =  \sup_{p \in \h}{ \left\{ \left\langle  x,p \right\rangle  -f^*(p) -\frac{\rho}{2} \left\| p \right\|^2 \right\}}.
\end{align*}
As already seen, $^{\rho}f$ is differentiable and its gradient (cf. \eqref{grad} and \eqref{prox-f-star})
$$\nabla(^{\rho}f) : \h \rightarrow \h, \ \nabla(^{\rho}f) = \tfrac{1}{\rho}(x - \Prox_{\rho f}(x)) = \Prox_{\frac{1}{\rho}f^*}\left( \frac{x}{\rho}\right) \ \forall x \in \h,$$
is $\tfrac{1}{\rho}$-Lipschitz continuous.

For $\mu \in \R_{++}$ we smooth $g \circ K$ via $^{\mu}g \circ K : \h \rightarrow \R$, $^{\mu}g \circ K(x) = \left(g\Box\frac{1}{2\mu}\left\| \cdot \right\|^2\right)(Kx)$ for every $x \in \h$.
According to the Fenchel-Moreau Theorem and due to \eqref{exinfconv}, one has
\begin{align*}
^{\mu}g \circ K(x) & = \left(g^{**}\Box\frac{1}{2\mu}\left\| \cdot \right\|^2\right)(Kx)  = \left(g^* + \frac{\mu}{2}\left\| \cdot \right\|^2 \right)^*(Kx)\\
& =  \sup_{p \in \hk}{ \left\{ \left\langle x,K^*p \right\rangle -g^*(p) -\frac{\mu}{2} \left\| p \right\|^2 \right\}} \ \forall x \in \h.
\end{align*}
The function $^{\mu}g \circ K$ is differentiable and its gradient $\nabla(^{\mu}g \circ K) : \h \rightarrow \h$ fulfills (cf. \eqref{grad} and \eqref{prox-f-star})
$$\nabla(^{\mu}g \circ K)(x) = K^*\nabla(^{\mu}g)(Kx) = \tfrac{1}{\mu}K^*(Kx - \Prox_{\mu g}(Kx)) = K^*\Prox_{\tfrac{1}{\mu}g^*}\left(\tfrac{Kx}{\mu} \right) \ \forall x \in \h.$$
Further, for every $x,y \in \h$ it holds (see \eqref{firmnon})
\begin{align*}
	\left\| \nabla(^{\mu}g \circ K)(x) - \nabla(^{\mu}g \circ K)(y) \right\|
	&\leq \tfrac{1}{\mu}\|K\| \left\| (Kx - \Prox_{\mu g}(Kx)) - (Ky - \Prox_{\mu g}(Ky)) \right\|\\
	& \leq \frac{\left\| K \right\|^2}{\mu} \left\| x-y \right\|,
\end{align*}
which shows that $\nabla(^{\mu}g \circ K)$ is $\frac{\left\| K \right\|^2}{\mu}$-Lipschitz continuous.

Finally, we consider as smoothing function for $f + g \circ K$ the function $F^{\rho,\mu} : \h \rightarrow \R$, $F^{\rho,\mu}(x) = {^\rho} f(x) + {^\mu}g\circ K(x),$
which is differentiable with Lipschitz continuous gradient $\nabla F^{\rho,\mu} : \h \rightarrow \h$ given by
$$\nabla F^{\rho,\mu}(x) = \Prox_{\tfrac{1}{\rho}f^*}\left(\tfrac{x}{\rho}\right) + K^*\Prox_{\tfrac{1}{\mu}g^*}\left(\tfrac{Kx}{\mu} \right) \ \forall x \in \h,$$
having as Lipschitz constant $L(\rho,\mu) := \frac{1}{\rho} + \frac{\left\| K \right\|^2}{\mu}$.

For $\rho_2 \geq \rho_1 > 0$ and every $x \in \h$ it holds (cf. \eqref{dom-bound})
\begin{align*}
{^{\rho_1}f(x)}  & = \sup_{p \in \dom f^*}{ \left\{ \left\langle  x,p \right\rangle  -f^*(p)-\frac{\rho_1}{2} \left\| p \right\|^2\right\}}\\
               & \leq \sup_{p \in \dom f^*}{ \left\{ \left\langle  x,p \right\rangle -f^*(p) -\frac{\rho_2}{2} \left\| p \right\|^2 \right\}} + \sup_{p \in \dom f^*}\left\{ \frac{\rho_2 - \rho_1}{2}\left\| p \right\|^2\right\}\\
               & \leq {}^{\rho_2}f(x) + (\rho_2 - \rho_1)\frac{L_f^2}{2},
\end{align*}
which yields, letting $\rho_1 \downarrow 0$ (cf. \cite[Proposition 12.32]{BauschkeCombettes11}),
$${}^{\rho_2}f(x) \leq f(x) \leq {}^{\rho_2}f(x) + \rho_2\frac{L_f^2}{2}.$$
Similarly, for $\mu_2 \geq \mu_1 > 0$ and every $y \in \hk$ it holds
\begin{align*}
{^{\mu_1} g(y)} \leq {}^{\mu_2} g(y) + (\mu_2 - \mu_1)\frac{L_g^2}{2},
\end{align*}
and
$${}^{\mu_2}g(y) \leq g(y) \leq {}^{\rho_2}g(y) + \rho_2\frac{L_g^2}{2}.$$
Consequently, for $\rho_2 \geq \rho_1 > 0$, $\mu_2 \geq \mu_1 > 0$ and every $x \in \h$ we have
\begin{align}\label{ineq-smoothing-different-parameter}
\begin{aligned}
F^{\rho_2,\mu_2}(x) \leq F^{\rho_1,\mu_1}(x)  \leq  F^{\rho_2,\mu_2}(x) + (\rho_2 - \rho_1)\frac{L_f^2}{2} + (\mu_2 - \mu_1)\frac{L_g^2}{2}
\end{aligned}
\end{align}
and
\begin{align}\label{ineq-smoothing-different-parameter2}
F^{\rho_2,\mu_2}(x) \leq F(x) \leq & F^{\rho_2,\mu_2}(x) + \rho_2\frac{L_f^2}{2} + \mu_2\frac{L_g^2}{2}.
\end{align}

\subsection{The variable smoothing and the constant smoothing algorithms}\label{subvarsmo}

Throughout this paper  $F: \h \rightarrow \R$, $F(x) = f(x) + g(Kx)$, will denote the objective function of $(P)$. The variable smoothing algorithm which we present at the beginning of this subsection can be seen as an extension of the accelerated gradient method of Nesterov (cf. \cite{Nesterov83}) by using variable smoothing parameters, which we update in each iteration.
\begin{empheq}[box=\fbox]{align}\label{opt-scheme-Nesterov-variable-smoothing}
\text{Initialization} : & \ t_1=1, \ y_1=x_0 \in \h,\ (\rho_k)_{k\geq1},(\mu_k)_{k\geq1}\subseteq \R_{++}  \tag{A1} \\ %a,b>0,
\text{For} \ k\geq1 : & \  L_k=\frac{1}{\rho_k} + \frac{\left\| K \right\|^2}{\mu_k}, \notag \\ %\rho_k=\frac{1}{ak},\ \mu_k=\frac{1}{bk}, \
		& \ x_k = y_k - \frac{1}{L_k} \left( \Prox_{\frac{1}{\rho_k}f^*}\left( \frac{y_k}{\rho_k}\right) + K^*\Prox_{\frac{1}{\mu_k}g^*}\left( \frac{Ky_k}{\mu_k} \right) \right), \notag \\
		&  \ t_{k+1} = \frac{1+\sqrt{1+4t_k^2}}{2}, \notag \\
		&  \ y_{k+1} = x_k + \frac{t_k-1}{t_{k+1}}(x_k - x_{k-1}) \notag
\end{empheq}
The convergence of the algorithm \eqref{opt-scheme-Nesterov-variable-smoothing} is proved by the following theorem.

\begin{theorem}\label{assertion-convergence-variable-smoothing}
Let $f:\h \rightarrow \R$ be a convex and $L_f$-Lipschitz continuous function, $g:\hk \rightarrow \R$ a convex and $L_g$-Lipschitz continuous function, $K: \h \rightarrow \hk$ a linear continuous operator and $x^*\in\h$ an optimal solution to $(P)$. Then, when choosing
$$\rho_k = \frac{1}{ak} \ \mbox{and} \ \mu_k= \frac{1}{bk} \ \forall k \geq 1,$$
where $a,b \in \R_{++}$, algorithm \eqref{opt-scheme-Nesterov-variable-smoothing} generates a sequence $(x_k)_{k\geq1} \subseteq \h$ satisfying
\begin{align}\label{scheme2-upper-bound}
	F(x_{k+1})-F(x^*)
	\leq \frac{2(a + b\left\|K\right\|^2)}{k+2} \left\|x_0 - x^*\right\|^2 + \frac{2(1+\ln(k+1))}{k+2} \left( \frac{L^2_{f}}{a} + \frac{L^2_{g}}{b} \right) \ \forall k \geq 1,
\end{align}
thus yielding a rate of convergence for the objective of order $\O(\tfrac{\ln k}{k})$.
\end{theorem}
\begin{proof}
For any $k \geq 1$ we denote $F^{k}:=F^{\rho_{k},\mu_{k}}$, $p_k:=(t_k-1)(x_{k-1} - x_k)$ and
$$\xi_k:= \nabla F^{k}(y_k) = \Prox_{\frac{1}{\rho_k}f^*}\left( \frac{y_k}{\rho_k}\right) + K^*\Prox_{\frac{1}{\mu_k}g^*}\left( \frac{Ky_k}{\mu_k} \right).$$
For any $k \geq 1$ it holds
\begin{align*}
	p_{k+1} -x_{k+1}
	&= (t_{k+1} - 1)(x_k-x_{k+1}) - x_{k+1} \\
	&= (t_{k+1} - 1)x_k - t_{k+1} \left(y_{k+1} - \frac{1}{L_{k+1}} \nabla F^{k+1}(y_{k+1}) \right) \\
	&= p_k - x_k +  \frac{t_{k+1}}{L_{k+1}} \nabla F^{k+1}(y_{k+1})
\end{align*}
and from here it follows
$$\left\| p_{k+1} -x_{k+1} +x^*  \right\|^2$$
$$= \left\| p_k - x_k +x^*  \right\|^2 + 2\<p_k - x_k +x^*, \frac{t_{k+1}}{L_{k+1}} \xi_{k+1}  \> + \left(\frac{t_{k+1}}{L_{k+1}}\right)^2 \left\| \xi_{k+1} \right\|^2$$
$$=\left\| p_k - x_k +x^*  \right\|^2 + \frac{2t_{k+1}}{L_{k+1}}\<p_k,  \xi_{k+1} \>$$
$$+\frac{2t_{k+1}}{L_{k+1}}\<x^*-y_{k+1}-\frac{p_k}{t_{k+1}},  \xi_{k+1} \> + \left(\frac{t_{k+1}}{L_{k+1}}\right)^2 \left\| \xi_{k+1} \right\|^2$$
$$=\left\| p_k - x_k +x^*  \right\|^2 + \frac{2(t_{k+1}-1)}{L_{k+1}}\<p_k,  \xi_{k+1} \> +\frac{2t_{k+1}}{L_{k+1}}\<x^*-y_{k+1},  \xi_{k+1}  \> + \left(\frac{t_{k+1}}{L_{k+1}}\right)^2 \left\| \xi_{k+1} \right\|^2.$$
Further, using \eqref{ineq-nest}, since $x_{k+1}=y_{k+1}-\frac{1}{L_{k+1}}\xi_{k+1}$, it follows
	\begin{align}\label{ineq-gradient-norm}
		F^{k+1}(x_{k+1})
		& \leq F^{k+1}(y_{k+1})  + \< \xi_{k+1},x_{k+1}-y_{k+1}\> + \frac{L_{k+1}}{2}\left\| x_{k+1}-y_{k+1} \right\|^2 \notag \\
		&= F^{k+1}(y_{k+1}) -\frac{1}{L_{k+1}} \left\| \xi_{k+1} \right\|^2 + \frac{1}{2L_{k+1}} \left\| \xi_{k+1} \right\|^2 \notag \\
		&= F^{k+1}(y_{k+1}) - \frac{1}{2L_{k+1}} \left\| \xi_{k+1} \right\|^2
	\end{align}
and, from here, by making use of the convexity of $F^{k+1}$, we have
\begin{align}\label{ineq1}
	\< x^* - y_{k+1}, \xi_{k+1}\> &\leq F^{k+1}(x^*) - F^{k+1}(y_{k+1}) \notag \\
	&\overset{\mathclap{\eqref{ineq-gradient-norm}}}{\leq} F^{k+1}(x^*) - F^{k+1}(x_{k+1}) - \frac{1}{2L_{k+1}} \left\| \xi_{k+1} \right\|^2 \ \forall k \geq 1.
\end{align}
On the other hand, since $F^{k+1}(x_{k}) - F^{k+1}(y_{k+1}) \geq \< \xi_{k+1}, x_k - y_{k+1} \>$, we obtain
\begin{align}\label{ineq2}
	\left\| \xi_{k+1} \right\|^2 &\overset{\mathclap{\eqref{ineq-gradient-norm}}}{\leq} 2L_{k+1} (F^{k+1}(y_{k+1}) - F^{k+1}(x_{k+1})) \notag \\
	&\leq 2L_{k+1} \left(F^{k+1}(x_{k}) - F^{k+1}(x_{k+1}) -\frac{1}{t_{k+1}} \< \xi_{k+1},p_k \> \right) \ \forall k \geq 1.
\end{align}
Thus, as $t_{k+1}^2-t_{k+1}=t_k^2$ and by making use of \eqref{ineq-smoothing-different-parameter}, for any $k \geq 1$ it yields
\begin{align*}
	&\left\| p_{k+1} -x_{k+1} +x^*  \right\|^2 - \left\| p_k - x_k +x^*  \right\|^2 \\
	&\overset{\mathclap{\eqref{ineq1}}}{\leq} \frac{2(t_{k+1}-1)}{L_{k+1}}\<p_k,  \xi_{k+1} \> +\frac{2t_{k+1}}{L_{k+1}}(F^{k+1}(x^*) - F^{k+1}(x_{k+1})) + \frac{t_{k+1}^2-t_{k+1}}{L_{k+1}^2} \left\| \xi_{k+1} \right\|^2 \\
	&\overset{\mathclap{\eqref{ineq2}}}{\leq} \frac{2t_{k+1}}{L_{k+1}}(F^{k+1}(x^*) - F^{k+1}(x_{k+1})) + \frac{2(t_{k+1}^2-t_{k+1})}{L_{k+1}}(F^{k+1}(x_k) - F^{k+1}(x_{k+1})) \\
	&\overset{\mathclap{\eqref{ineq-smoothing-different-parameter}}}{\leq} \frac{2t_k^2}{L_{k+1}}\left(F^{k}(x_k) - F^{k}(x^*) + (\rho_k - \rho_{k+1})\frac{L_f^2}{2} + (\mu_k - \mu_{k+1})\frac{L_g^2}{2}\right)\\
    & - \frac{2t_{k+1}^2}{L_{k+1}}(F^{k+1}(x_{k+1}) - F^{k+1}(x^*))\displaybreak\\
    & = \frac{2t_k^2}{L_{k+1}}\left(F^{k}(x_k) - F^{k}(x^*) + \rho_k \frac{L_f^2}{2} + \mu_k \frac{L_g^2}{2}\right) - \frac{2t_{k+1}^2}{L_{k+1}}(F^{k+1}(x_{k+1}) - F^{k+1}(x^*))\\
    & - \frac{2t_k^2}{L_{k+1}}\left(\rho_{k+1} \frac{L_f^2}{2} + \mu_{k+1}\frac{L_g^2}{2}\right).
\end{align*}
By using \eqref{ineq-smoothing-different-parameter2} it follows that for any $k \geq 1$
$$F^{k}(x_k) - F^{k}(x^*) + \rho_k \frac{L_f^2}{2} + \mu_k \frac{L_g^2}{2} \geq F(x_k) - F^k(x^*) \geq F(x_k) - F(x^*) \geq 0,$$
thus
\begin{align*}
	&\left\| p_{k+1} -x_{k+1} +x^*  \right\|^2 - \left\| p_k - x_k +x^*  \right\|^2 \\
    & \leq \frac{2t_k^2}{L_{k}}\left(F^{k}(x_k) - F^{k}(x^*) + \rho_k \frac{L_f^2}{2} + \mu_k \frac{L_g^2}{2}\right) - \frac{2t_{k+1}^2}{L_{k+1}}(F^{k+1}(x_{k+1}) - F^{k+1}(x^*))\\
    & - \frac{2t_k^2}{L_{k+1}}\left(\rho_{k+1} \frac{L_f^2}{2} + \mu_{k+1}\frac{L_g^2}{2}\right)\\
    & = \frac{2t_k^2}{L_{k}}\left(F^{k}(x_k) - F^{k}(x^*) + \rho_k \frac{L_f^2}{2} + \mu_k \frac{L_g^2}{2}\right) - \frac{2t_{k+1}^2}{L_{k+1}}(F^{k+1}(x_{k+1}) - F^{k+1}(x^*)) \\
    & - \frac{2t_{k+1}^2}{L_{k+1}} \left( \rho_{k+1} \frac{L_f^2}{2} + \mu_{k+1} \frac{L_g^2}{2} \right) + \frac{2t_{k+1}}{L_{k+1}} \left( \rho_{k+1} \frac{L_f^2}{2} + \mu_{k+1} \frac{L_f^g}{2} \right),
\end{align*}
which implies that
\begin{align*}
	&\left\| p_{k+1} -x_{k+1} +x^*  \right\|^2 + \frac{2t_{k+1}^2}{L_{k+1}}\left(F^{k+1}(x_{k+1}) - F^{k+1}(x^*) + \rho_{k+1} \frac{L_f^2}{2} + \mu_{k+1} \frac{L_g^2}{2} \right)\\
    & \leq \left\| p_k - x_k +x^*  \right\|^2 + \frac{2t_k^2}{L_{k}}\left(F^{k}(x_k) - F^{k}(x^*) +  \rho_k \frac{L_f^2}{2} + \mu_k \frac{L_g^2}{2} \right)\\
    & + \frac{2t_{k+1}}{L_{k+1}} \left( \rho_{k+1} \frac{L_f^2}{2} + \mu_{k+1} \frac{L_g^2}{2} \right).
\end{align*}
Making again use of \eqref{ineq-smoothing-different-parameter2} this further yields for any $k \geq 1$
\begin{align}\label{ineq-fval-estimate-variable-smoothing1}
	& \frac{2t_{k+1}^2}{L_{k+1}}\left( F(x_{k+1})-F(x^*) \right) \notag \\
	& \leq \frac{2t_{k+1}^2}{L_{k+1}}\left(F^{k+1}(x_{k+1})-F^{k+1}(x^*) + \rho_{k+1} \frac{L_f^2}{2} + \mu_{k+1} \frac{L_g^2}{2} \right) + \left\| p_{k+1} -x_{k+1} +x^*  \right\|^2 \notag \\
	& \leq \frac{2t_1^2}{L_{1}}\left(F^{1}(x_1) - F^{1}(x^*) + \rho_1 \frac{L_f^2}{2} + \mu_1 \frac{L_g^2}{2}\right) + \left\| p_1 -x_1 +x^*  \right\|^2 \notag \\
	& + \sum_{s=1}^{k}\frac{2t_{s+1}}{L_{s+1}} \left( \rho_{s+1} \frac{L_f^2}{2} + \mu_{s+1} \frac{L_g^2}{2} \right).
\end{align}
Since $x_1 = y_1 - \frac{1}{L_1} \nabla F^{1}(y_1)$ and
\begin{align*}
	F^{1}(x^*) &\geq F^{1}(y_1) + \< \nabla F^{1}(y_1) , x^*-y_1\> \\
	F^{1}(x_1) &\leq F^{1}(y_1) + \< \nabla F^{1}(y_1) , x_1-y_1 \> + \frac{L_1}{2} \left\| x_1-y_1\right\|^2,
\end{align*}
we get
\begin{align*}
    &\frac{2t^2_1}{L_{1}}\left(F^{1}(x_1) - F^{1}(x^*)\right) + \left\| p_1 -x_1 +x^*  \right\|^2 \notag \\
	& \leq  2 \langle x_1-y_1,x^*-y_1\rangle -\|x_1 - y_1\|^2  + \|x_1-x^*\|^2 = \|y_1-x^*\|^2 = \|x_0 - x^*\|^2
\end{align*}
and this, together with \eqref{ineq-fval-estimate-variable-smoothing1}, give rise to the following estimate
\begin{align}\label{ineq-fval-estimate-variable-smoothing}
\frac{2t_{k+1}^2}{L_{k+1}}\left( F(x_{k+1})-F(x^*) \right) \leq \left\|x_0 - x^*\right\|^2 + \sum_{s=1}^{k+1}\frac{t_{s}}{L_{s}} \left( \rho_{s}{L_f^2}+ \mu_{s}{L_g^2}\right).
\end{align}
Furthermore, since $t_{k+1} \geq \frac{1}{2} + t_k$ for any $k \geq 1$, it follows that $t_{k+1} \geq \frac{k+2}{2}$, which, along with the fact that
$L_k = \frac{1}{\rho_k} + \frac{\left\| K \right\|^2}{\mu_k} = (a + b\left\| K\right\|^2)k$, lead for any $k \geq 1$  to the following estimate
\begin{align*}
	& F(x_{k+1})-F(x^*) \\
	&\leq \frac{2(a + b\left\|K\right\|^2)(k+1)}{(k+2)^2} \left( \left\|x_0 - x^*\right\|^2 + L_f^2\sum_{s=1}^{k+1}\frac{t_{s}\rho_{s}}{L_{s}} +  L_g^2\sum_{s=1}^{k+1}\frac{t_{s}\mu_{s}}{L_{s}} \right) \\
	&\leq \frac{2(a + b\left\|K\right\|^2)}{k+2} \left\| x_0 - x^*\right\|^2 + \frac{2}{k+2}\sum_{s=1}^{k+1}\frac{t_{s}}{s^2} \left (\frac{L_f^2}{a}  + \frac{L_f^2}{b} \right).
\end{align*}
Using now that $t_{k+1} \leq 1+t_k$ for any $k \geq 1$, it yields that $t_{k+1} \leq k+1$ for any $k \geq 0$, thus
\begin{align*}
	\sum_{s=1}^{k+1}\frac{t_{s}}{s^2} \leq \sum_{s=1}^{k+1}\frac{1}{s}
	\leq 1 + \sum_{s=2}^{k+1}\int_{s-1}^s{\frac{1}{x}\,\mathrm{d}x} = 1 + \int_{1}^{k+1}{\frac{1}{x}\,\mathrm{d}x} = 1+ \ln(k+1).
\end{align*}
Finally, we obtain that
\begin{align*}
	F(x_{k+1})-F(x^*)
	\leq \frac{2(a + b\left\|K\right\|^2)}{k+2} \left\| x_0 - x^*\right\|^2 + \frac{2(1+\ln(k+1))}{k+2} \left(\frac{L_f^2}{a} + \frac{L_g^2}{b} \right) \ \forall k \geq 1,
\end{align*}
which concludes the proof.
\end{proof}
In the second part of this subsection we propose a variant of algorithm \eqref{opt-scheme-Nesterov-variable-smoothing} formulated with  constant smoothing parameters:
\begin{empheq}[box=\fbox]{align}\label{opt-scheme-Nesterov}
\text{Initialization} : & \ t_1=1, \ y_1=x_0 \in \h, \ \rho,\mu \in \R_{++}, \tag{A2} \\
                       &  \ L(\rho,\mu)=\frac{1}{\rho} + \frac{\left\| K \right\|^2}{\mu} \notag\\
\text{For} \ k \geq 1 : & \ x_k = y_k - \frac{1}{L(\rho,\mu)} \left( \Prox_{\frac{1}{\rho}f^*}\left( \frac{y_k}{\rho}\right) + K^*\Prox_{\frac{1}{\mu}g^*}\left( \frac{Ky_k}{\mu} \right) \right), \notag \\
		               &  \ t_{k+1} = \frac{1+\sqrt{1+4t_k^2}}{2}, \notag \\
		               &  \ y_{k+1} = x_k + \frac{t_k-1}{t_{k+1}}(x_k - x_{k-1}) \notag
\end{empheq}
Constant smoothing parameters have been also used in \cite{DevGliNes12} and \cite{BotHendrich12,BotHendrich12b} within the framework of double smoothing algorithms, which assume the regularization in two steps of the Fenchel dual problem to $(P)$ and, consequently, the solving of an unconstrained optimization problem with a strongly convex and differentiable objective function having a Lipschitz continuous gradient.

\begin{theorem}\label{convergence-constant-smoothing}
Let $f:\h \rightarrow \R$ be a convex and $L_f$-Lipschitz continuous function, $g:\hk \rightarrow \R$ a convex and $L_g$-Lipschitz continuous function, $K: \h \rightarrow \hk$ a linear continuous operator and $x^*\in\h$ an optimal solution to $(P)$. Then, when choosing for $\varepsilon>0$
$$\rho=\frac{2\varepsilon}{3L_f^2} \ \mbox{and} \ \mu=\frac{2\varepsilon}{3L_g^2},$$
algorithm \eqref{opt-scheme-Nesterov} generates a sequence $(x_k)_{k\geq1} \subseteq \h$ which provides an $\varepsilon$-optimal solution to $(P)$ with a rate of convergence for the objective of order $\O(\tfrac{1}{k})$.
\end{theorem}
\begin{proof}
In order to prove this statement, one has only to reproduce the first part of the proof of Theorem \ref{assertion-convergence-variable-smoothing} when
$$\rho_k=\rho, \mu_k=\mu \ \mbox{and} \ L_k=L(\rho,\mu)=\frac{1}{\rho} + \frac{\left\| K \right\|^2}{\mu} \ \forall k \geq 1,$$
fact which leads to \eqref{ineq-fval-estimate-variable-smoothing}. This inequality reads in this particular situation
\begin{align*}
	F(x_{k+1}) - F(x^*) & \leq \frac{L(\rho,\mu)\left\|x_0 - x^* \right\|^2}{2t_{k+1}^2} + \frac{\rho L_f^2 + \mu L_g^2}{2t_{k+1}^2}\sum_{s=1}^{k+1} t_s \ \forall k \geq 1.
\end{align*}
Since $t_{k+1}^2 = t_k^2 + t_{k+1}$ for any $k \geq 1$, one can inductively prove that $t_{k+1}^2 = \sum_{s=1}^{k+1}{t_s}$, which, together with the fact that $t_{k+1} \geq \frac{k+2}{2}$ for any $k \geq 1$, yields
\begin{align*}
	F(x_{k+1}) - F(x^*) \leq \frac{2L(\rho,\mu)\left\|x_0 - x^* \right\|^2}{(k+2)^2} + \frac{\rho L_f^2 + \mu L_g^2}{2} \ \forall k \geq 1.
\end{align*}
In order to obtain $\varepsilon$-optimality for the objective of the problem $(P)$, where $\varepsilon >0$ is a given level of accuracy, we choose $\rho=\frac{2\varepsilon}{3L_f^2}$ and $\mu=\frac{2\varepsilon}{3L_g^2}$ and, thus, we have only to force the first term in the right-hand side of the above estimate to be less than or equal to $\frac{\varepsilon}{3}$. Taking also into account that in this situation $L(\rho,\mu)=\frac{3L_f^2 + 3L_g^2\|K\|^2}{2\varepsilon}$, it holds
\begin{align*}
	\frac{\varepsilon}{3} &\geq \frac{2L(\rho,\mu)\left\|x_0 - x^* \right\|^2}{(k+2)^2} = \frac{3\left(L_f^2 + L_g^2\|K\|^2 \right)\left\|x_0 - x^* \right\|^2}{\varepsilon(k+2)^2} \\
	\Leftrightarrow \frac{\varepsilon^2}{9} &\geq \frac{\left(L_f^2 + L_g^2\|K\|^2 \right)\left\|x_0 - x^* \right\|^2}{(k+2)^2} \\
	\Leftrightarrow \frac{\varepsilon}{3} &\geq \frac{\sqrt{L_f^2 + L_g^2\|K\|^2} \left\|x_0 - x^* \right\|}{k+2},
\end{align*}
which shows that an $\varepsilon$-optimal solution to $(P)$ can be provided with a rate of convergence for the objective of order $\O(\tfrac{1}{k})$.
\end{proof}
The rate of convergence of algorithm \eqref{opt-scheme-Nesterov-variable-smoothing} may not be as good as the one proved for the algorithm with constant smoothing parameters depending on a fixed level of accuracy $\varepsilon > 0$.  However, the main advantage of the variable smoothing methods is given by the fact that the sequence of objective values $(f(x_k) + g(Kx_k))_{k \geq 1}$ converges to the optimal objective value of $(P)$, whereas, when generated by algorithm \eqref{opt-scheme-Nesterov}, despite of the fact that it approximates the optimal objective value with a better convergence rate, this sequence may not converge to this.

\subsection{The case when $f$ is differentiable with Lipschitz continuous gradient}\label{subconstsmo}

In this subsection we show how the algorithms \eqref{opt-scheme-Nesterov-variable-smoothing} and \eqref{opt-scheme-Nesterov} for solving the problem $(P)$ can be adapted to the situation when $f$ is a differentiable function with Lipschitz continuous gradient. We provide iterative schemes with variable and constant smoothing variables and corresponding convergence statements. More precisely, we deal with the optimization problem
\begin{equation*}
\hspace{-1.8cm}(P) \quad \quad \inf_{x \in \h}{\left\{f(x)+g(Kx)\right\}},
\end{equation*}
where $K:\h \rightarrow \hk$ is a linear continuous operator, $f:\h \rightarrow \R$ is a convex and differentiable function with $L_{\nabla f}$-Lipschitz continuous gradient and $g:\hk \rightarrow \R$ is a convex and $L_g$-Lipschitz continuous function.

Algorithm \eqref{opt-scheme-Nesterov-variable-smoothing} can be adapted to this framework as follows:
\begin{empheq}[box=\fbox]{align}\label{diff-scheme-variable-smoothing}
\text{Initialization} : & \ t_1=1, \ y_1=x_0 \in \h,\ (\mu_k)_{k\geq1} \subseteq \R_{++}  \tag{A3} \\
\text{For} \ k \geq 1 : & \  L_k= L_{\nabla f} + \frac{\left\| K \right\|^2}{\mu_k}, \notag \\
		& \ x_k = y_k - \frac{1}{L_k} \left(\nabla f(y_k) + K^*\Prox_{\frac{1}{\mu_k}g^*}\left( \frac{Ky_k}{\mu_k} \right) \right), \notag \\
		& \ t_{k+1} = \frac{1+\sqrt{1+4t_k^2}}{2}, \notag \\
		& \ y_{k+1} = x_k + \frac{t_k-1}{t_{k+1}}(x_k - x_{k-1}) \notag
\end{empheq}
while its convergence is furnished by the following theorem.

\begin{theorem}\label{diff-convergence-variable-smoothing}
Let $f:\h \rightarrow \R$ be a convex and differentiable function with $L_{\nabla f}$-Lipschitz continuous gradient, $g:\hk \rightarrow \R$ a convex and $L_g$-Lipschitz continuous function,  $K: \h \rightarrow \hk$ a nonzero linear continuous operator and $x^*\in\h$ an optimal solution to $(P)$. Then, when choosing
$$\mu_k= \frac{1}{bk} \ \forall k \geq 1,$$
where $b \in \R_{++}$, algorithm \eqref{diff-scheme-variable-smoothing} generates a sequence $(x_k)_{k\geq1} \subseteq \h$ satisfying for any $k \geq 1$
\begin{align}\label{diff-scheme-upper-bound}
	F(x_{k+1})-F(x^*) & \leq \frac{2(L_{\nabla f} + b\left\|K\right\|^2)}{k+2}\left\|x_0 - x^*\right\|^2  + \frac{2(1+\ln(k+1))}{k+2}\frac{L_g^2(L_{\nabla f} + b\left\|K\right\|^2)}{b^2\left\|K\right\|^2},
\end{align}
thus yielding a rate of convergence for the objective of order $\O(\tfrac{\ln k}{k})$.
\end{theorem}
\begin{proof}
For any $k \geq 1$ we denote by $F^{k} : \h \rightarrow \R$, $F^k(x)=f(x) + {}^{\mu_k}g(Kx)$. For any $k \geq 1$ and every $x \in \h$ it holds $\nabla F^{k}(x) = \nabla f(x)+ K^*\Prox_{\frac{1}{\mu_k}g^*}\left( \frac{Kx}{\mu_k} \right)$ and $\nabla F^k$ is $L_k$-Lipschitz continuous, where $L_k=L_{\nabla f} + \frac{\left\| K \right\|^2}{\mu_k}$.

As in  the proof of Theorem \ref{assertion-convergence-variable-smoothing}, by defining $p_k:=(t_k-1)(x_{k-1} - x_k)$, we obtain for any $k \geq 1$
\begin{align*}
	&\left\| p_{k+1} -x_{k+1} +x^*  \right\|^2 - \left\| p_k - x_k +x^*  \right\|^2 \\
	& \leq \frac{2t_k^2}{L_{k+1}}\left (F^{k+1}(x_{k}) - F^{k+1}(x^*) \right) - \frac{2t_{k+1}^2}{L_{k+1}}(F^{k+1}(x_{k+1}) - F^{k+1}(x^*))\\
    &\leq \frac{2t_k^2}{L_{k+1}}\left (F^{k}(x_k) - F^{k+1}(x^*) + (\mu_k - \mu_{k+1})\frac{L_g^2}{2} \right) - \frac{2t_{k+1}^2}{L_{k+1}}(F^{k+1}(x_{k+1}) - F^{k+1}(x^*))\\
    & \leq \frac{2t_k^2}{L_{k+1}}\left (F^{k}(x_k) - F^{k}(x^*) + \mu_k \frac{L_g^2}{2} \right ) - \frac{2t_{k+1}^2}{L_{k+1}}(F^{k+1}(x_{k+1}) - F^{k+1}(x^*)) - \frac{t_k^2}{L_{k+1}}\mu_{k+1}L_g^2\\
    & \leq  \frac{2t_k^2}{L_{k}}\left (F^{k}(x_k) - F^{k}(x^*) + \mu_k \frac{L_g^2}{2} \right ) - \frac{2t_{k+1}^2}{L_{k+1}}(F^{k+1}(x_{k+1}) - F^{k+1}(x^*)) - \frac{t_k^2}{L_{k+1}}\mu_{k+1}L_g^2\\
    & = \frac{2t_k^2}{L_{k}}\left (F^{k}(x_k) - F^{k}(x^*) + \mu_k \frac{L_g^2}{2}  \right ) - \frac{2t_{k+1}^2}{L_{k+1}}(F^{k+1}(x_{k+1}) - F^{k+1}(x^*))\\
	&- \frac{t_{k+1}^2L_g^2}{L_{k+1}}\mu_{k+1} + \frac{t_{k+1}L_g^2}{L_{k+1}}\mu_{k+1}
\end{align*}
and, consequently,
\begin{align*}
	&\left\| p_{k+1} -x_{k+1} +x^*  \right\|^2 + \frac{2t_{k+1}^2}{L_{k+1}}\left(F^{k+1}(x_{k+1}) - F^{k+1}(x^*) + \mu_{k+1} \frac{L_g^2}{2} \right)\\
    & \leq \left\| p_k - x_k +x^*  \right\|^2 + \frac{2t_k^2}{L_{k}}\left(F^{k}(x_k) - F^{k}(x^*) + \mu_k \frac{L_g^2}{2} \right) + \frac{t_{k+1}L_g^2}{L_{k+1}}\mu_{k+1}.
\end{align*}
For any $k \geq 1$ it holds
\begin{align*}
	& \frac{2t_{k+1}^2}{L_{k+1}}\left( F(x_{k+1})-F(x^*) \right) \notag \\
	& \leq \frac{2t_{k+1}^2}{L_{k+1}}\left(F^{k+1}(x_{k+1})-F^{k+1}(x^*) + \mu_{k+1} \frac{L_g^2}{2} \right) + \left\| p_{k+1} -x_{k+1} +x^*  \right\|^2
\notag \\
	& \leq \frac{2t_1^2}{L_{1}}\left(F^{1}(x_1) - F^{1}(x^*) + \mu_1 \frac{L_g^2}{2}\right) + \left\| p_1 -x_1 +x^*  \right\|^2 \notag \\
	& + \sum_{s=1}^{k}\frac{t_{s+1}L_g^2}{L_{s+1}}\mu_{s+1},
\end{align*}
which yields
\begin{align}\label{diff-estimate-variable-smoothing}
\frac{2t_{k+1}^2}{L_{k+1}}\left( F(x_{k+1})-F(x^*) \right) \leq \left\|x_0 - x^*\right\|^2 + \sum_{s=1}^{k+1}\frac{t_{s}L_g^2}{L_{s}}\mu_{s}.
\end{align}
For any $k \geq 1$, since $t_{k+1} \geq \frac{k+2}{2}$ and
$L_k = L_{\nabla f} + \frac{\left\| K \right\|^2}{\mu_k} = L_{\nabla f} + b\left\|K\right\|^2k$, it follows
\begin{align*}
	& F(x_{k+1})-F(x^*) \\
	&\leq \frac{2(L_{\nabla f} + b\left\|K\right\|^2(k+1))}{(k+2)^2} \left( \left\|x_0 - x^*\right\|^2 + \sum_{s=1}^{k+1}\frac{t_{s}L_g^2}{(L_{\nabla f} + b\left\|K\right\|^2s)sb} \right).
\end{align*}
Thus, for any $k \geq 1$, since $t_k \leq k$, it yields
\begin{align*}
	& F(x_{k+1})-F(x^*) \\
	&\leq \frac{2(L_{\nabla f} + b\left\|K\right\|^2(k+1))}{(k+2)^2} \left( \left\|x_0 - x^*\right\|^2 + \sum_{s=1}^{k+1}\frac{L_g^2}{(L_{\nabla f} + b\left\|K\right\|^2s)b} \right)\\
&\leq \frac{2(L_{\nabla f} + b\left\|K\right\|^2(k+1))}{(k+2)^2} \left( \left\|x_0 - x^*\right\|^2 + \sum_{s=1}^{k+1}\frac{L_g^2}{b^2\left\|K\right\|^2s} \right)
\end{align*}
\begin{align*}
    &\leq \frac{2(L_{\nabla f} + b\left\|K\right\|^2(k+1))}{(k+2)^2} \left( \left\|x_0 - x^*\right\|^2 + \frac{L_g^2}{b^2\left\|K\right\|^2} (1+\ln(k+1)) \right)\\
    &\leq \frac{2(L_{\nabla f} + b\left\|K\right\|^2)}{k+2} \left( \left\|x_0 - x^*\right\|^2 + \frac{L_g^2}{b^2\left\|K\right\|^2} (1+\ln(k+1)) \right)\\
     &\leq \frac{2(L_{\nabla f} + b\left\|K\right\|^2)}{k+2}\left\|x_0 - x^*\right\|^2  + \frac{2(1+\ln(k+1))}{k+2}\frac{L_g^2(L_{\nabla f} + b\left\|K\right\|^2)}{b^2\left\|K\right\|^2}.
    \end{align*}
\end{proof}

By adapting \eqref{diff-scheme-variable-smoothing} to the framework considered in this subsection we obtain the following algorithm with constant smoothing variables:
\begin{empheq}[box=\fbox]{align}\label{diff-scheme-constant-smoothing}
\text{Initialization} : & \ t_1=1, \ y_1=x_0 \in \h, \ \mu \in \R_{++}, \tag{A4} \\
                       &  \ L(\mu)=L_{\nabla f} + \frac{\left\| K \right\|^2}{\mu} \notag\\
\text{For }	k \geq 1 :  & \ x_k = y_k - \frac{1}{L(\mu)} \left(\nabla f(y_k) + K^*\Prox_{\frac{1}{\mu}g^*}\left( \frac{Ky_k}{\mu} \right) \right), \notag \\
		               &  \ t_{k+1} = \frac{1+\sqrt{1+4t_k^2}}{2}, \notag \\
		               &  \ y_{k+1} = x_k + \frac{t_k-1}{t_{k+1}}(x_k - x_{k-1}) \notag
\end{empheq}
The convergence of algorithm \eqref{diff-scheme-constant-smoothing} is stated by the following theorem, which can be proved in the lines of the proof of Theorem \ref{diff-convergence-variable-smoothing}.

\begin{theorem}\label{diff-convergence-constant-smoothing}
Let $f:\h \rightarrow \R$ be a convex and differentiable function with $L_{\nabla f}$-Lipschitz continuous gradient, $g:\hk \rightarrow \R$ a convex and $L_g$-Lipschitz continuous function,  $K: \h \rightarrow \hk$ a nonzero linear continuous operator and $x^*\in\h$ an optimal solution to $(P)$. Then, when choosing for $\varepsilon>0$
$$\mu=\frac{\varepsilon}{L_g^2},$$
algorithm \eqref{diff-scheme-constant-smoothing} generates a sequence $(x_k)_{k\geq1} \subseteq \h$ which provides an $\varepsilon$-optimal solution to $(P)$ with a rate of convergence for the objective of order $\O(\tfrac{1}{k})$.
\end{theorem}

\subsection{The optimization problem with the sum of more than two functions in the objective}\label{submore}

We close this section by discussing the employment of the algorithmic schemes presented in the previous two subsections to the optimization problem \eqref{eq:primal-sum}
\begin{equation*}
\inf_{x \in \h}{\left\{ f(x) + \sum_{i=1}^m{g_i(K_ix)}\right\}},
\end{equation*}
where $\h$ and $\hk_i$, $i=1,...,m$, are real Hilbert spaces, $f : \h \rightarrow \R$ is a convex and either $L_f$-Lipschitz continuous or differentiable with $L_{\nabla f}$-continuous gradient function, $g_i : \hk_i \rightarrow \R$ are convex and $L_{g_i}$-Lipschitz continuous functions and $K_i : \h \rightarrow \hk_i$, $i=1,...,m$, are linear continuous operators. By endowing $\hk:= \hk_1 \times ... \times \hk_m$ with the inner product defined as
$$\langle y,z \rangle = \sum_{i=1}^m \langle y_i,z_i \rangle \ \forall y,z \in \hk,$$
and with the corresponding norm and by defining $g : \hk \rightarrow \R, g(y_1,...,y_m) = \sum_{i=1}^m g_i(y_i)$ and $K : \h \rightarrow \hk, Kx = (K_1x,...,K_mx)$, problem \eqref{eq:primal-sum} can be equivalently written as
\begin{equation*}
\inf_{x \in \h}{\left\{ f(x) + g(Kx)\right\}}
\end{equation*}
and, consequently, solved via one of the variable or constant smoothing algorithms introduced in the subsections \ref{subvarsmo} and \ref{subconstsmo}, depending on the properties the function $f$ is endowed with.

In the following we determine the elements related to the above constructed function $g$ which appear in these iterative schemes and in the corresponding convergence statements. Obviously, the function $g$ is convex and, since
for every $(y_1,...,y_m), (z_1,...,z_m) \in \hk$
$$|g(y_1,...,y_m) - g(z_1,...,z_m)| \leq \sum_{i=1}^m L_{g_i}\|y_i-z_i\| \leq \left (\sum_{i=1}^m L^2_{g_i}\right)^{\frac{1}{2}}\|(y_1,...,y_m) - (z_1,...,z_m)\|,$$
it is $\left (\sum_{i=1}^m L^2_{g_i}\right)^{\frac{1}{2}}$-Lipschitz continuous. On the other hand, for each $\mu \in \R_{++}$ and $(y_1,...,y_m) \in \hk$ it holds
$${}^\mu g(y_1,...,y_m) = \sum_{i=1}^m {}^\mu g_i(y_i),$$
thus
\begin{align*}
\nabla({}^\mu g)(y_1,...,y_m) & = \left(\nabla({}^\mu g_1)(y_1),...,\nabla({}^\mu g_m)(y_m) \right)\\
                              & = \left(\Prox_{\frac{1}{\mu}g_i^*}\left(\frac{y_1}{\mu}\right),...,\Prox_{\frac{1}{\mu}g_m^*}\left( \frac{y_m}{\mu}\right) \right).
\end{align*}
Since $K^*(y_1,...,y_m) = \sum_{i=1}^m K_i^*y_i$, for every $(y_1,...,y_m) \in \hk$, we have
\begin{align*}
\nabla({}^\mu g \circ K)(x) = K^*\nabla({}^\mu g)(K_1x,...,K_mx) & = \sum_{i=1}^m K_i^*\nabla({}^\mu g_i)(K_ix)\\
                                                                 & =  \sum_{i=1}^m K_i^*\Prox_{\frac{1}{\mu}g_i^*} \left( \frac{K_ix}{\mu} \right) \ \forall x \in \h.
\end{align*}
Finally, we notice that for arbitrary $x,y \in \h$ one has
\begin{align*}
	\left\| \nabla({}^\mu g \circ K)(x) - \nabla({}^\mu g \circ K)(y) \right\|
	&= \left\| \sum_{i=1}^m K_i^*\nabla({}^\mu g_i)(K_ix) - \sum_{i=1}^m K_i^*\nabla({}^\mu g_i)(K_iy) \right\|  \\
	&\leq \sum_{i=1}^m \|K_i\| \left\|\nabla({}^\mu g_i)(K_ix) - \nabla({}^\mu g_i)(K_iy) \right\| \\
	&\leq \sum_{i=1}^m \frac{\left\| K_i \right\|}{\mu} \left\|K_ix - K_iy \right\| \leq \frac{\sum_{i=1}^m \left\| K_i \right\|^2}{\mu} \left\| x-y \right\|,
\end{align*}
which shows that the Lipschitz constant of $\nabla({}^\mu g \circ K)$ is $\frac{\sum_{i=1}^m \left\| K_i \right\|^2}{\mu}$.

\section{Numerical experiments}\label{sectionExample}
\subsection{Image processing}\label{subsectionImageProcessing}

The first numerical experiment involving the variable smoothing algorithm concerns the solving of an extremely ill-conditioned linear inverse problem which arises in the field of signal and image processing, by basically solving
the regularized nondifferentiable convex optimization problem
\begin{equation}\label{probimageproc}
\inf_{x \in \R^n}{\left\{ \left\| Ax-b \right\|_1 + \lambda \left\| Wx \right\|_1\right\}},
\end{equation}
where $b \in \R^n$ is the blurred and noisy image, $A: \R^n \rightarrow \R^n$ is a blurring operator, $W:\R^n \rightarrow \R^n$ is the discrete Haar wavelet transform with four levels and $\lambda > 0$ is the regularization parameter. The blurring operator is constructed by making use of the Matlab routines {\ttfamily imfilter} and {\ttfamily fspecial} as follows:
\lstset{language=Matlab}
\begin{lstlisting}[numbers=left,numberstyle=\tiny,frame=tlrb,showstringspaces=false]
H=fspecial('gaussian',9,4); % gaussian blur of size 9 times 9
                            % and standard deviation 4
B=imfilter(X,H,'conv','symmetric'); % B=observed blurred image
                                    % X=original image		
\end{lstlisting}
The function {\ttfamily fspecial} returns a rotationally symmetric Gaussian lowpass filter of size $9 \times 9$ with standard deviation $4$,  the entries of $H$ being nonnegative and their sum adding up to $1$. The function {\ttfamily imfilter} convolves the filter $H$  with the image $X$ and furnishes the blurred image $B$. The boundary option ``symmetric'' corresponds to reflexive boundary conditions. Thanks to the rotationally symmetric filter $H$, the linear operator $A$ defined via the routine {\ttfamily imfilter} is symmetric, too. By making use of the real spectral decomposition of $A$, it shows that $\left\| A \right\|^2=1$. Furthermore, since $W$ is an orthogonal wavelet, it holds $\left\|W \right\|^2=1$.

The optimization problem \eqref{probimageproc} can be written as
\begin{equation*}
\inf_{x \in \R^n}{\left\{f(x) + g_1(Ax) + g_2(Wx)\right\}},
\end{equation*}
where $f: \R^n \rightarrow \R$ is taking to be $f \equiv 0$ with the Lipschitz constant of its gradient $L_{\nabla f} = 0$, $g_1:\R^n \rightarrow \R$, $g_1(y)=\left\| y-b \right\|_1$ is convex and $\sqrt{n}$-Lipschitz continuous and $g_2:\R^n \rightarrow \R$, $g_2(y)=\lambda \left\| y \right\|_1$ is convex and $\lambda \sqrt{n}$-Lipschitz continuous. For every $p \in \R^n$ it holds $g_1^*(p)=\delta_{\left[ -1, 1\right]^n}(p)+p^Tb$ and $g_2^*(p)=\delta_{\left[ -\lambda, \lambda\right]^n}(p)$ (see, for instance, \cite{Bot10}). We solved this problem, by using also the considerations made in Subsection \ref{submore}, with algorithm \eqref{diff-scheme-variable-smoothing} and computed to this aim for $\mu \in \R_{++}$ and $x \in \R^n$
\begin{align*}
		\Prox_{\frac{1}{\mu}g_1^*}\left( \frac{Ax}{\mu} \right)
		&= \argmin_{p \in \R^n}{\left\{ \frac{1}{\mu} g_1^*(p) +\frac{1}{2}\left\| \frac{Ax}{\mu} - p \right\|^2  \right\}}
		\!= \!\argmin_{p \in \left[ -1, 1 \right]^n}{\left\{ \frac{1}{\mu} p^Tb + \frac{1}{2}\left\| \frac{Ax}{\mu} - p \right\|^2\right\}} \\
		&= \argmin_{p \in \left[ -1, 1 \right]^n}{\left\{ \frac{1}{2}\left\| \frac{Ax}{\mu} - p \right\|^2 -\left(\frac{Ax}{\mu}-p\right)^T\frac{b}{\mu} + \frac{\left\| b \right\|^2}{2\mu^2} -  \frac{\left\| b \right\|^2}{2\mu^2} + \frac{(Ax)^Tb}{\mu^2}\right\}} \\
		&= \argmin_{p \in \left[ -1, 1 \right]^n}{\left\{ \frac{1}{2}\left\| \frac{Ax-b}{\mu} - p \right\|^2 \right\}} -  \frac{\left\| b \right\|^2}{2\mu^2} + \frac{(Ax)^Tb}{\mu^2}
		= \mathcal{P}_{\left[ -1, 1\right]^n}\left( \frac{Ax-b}{\mu} \right)
\end{align*}
and
\begin{align*}
		\Prox_{\frac{1}{\mu}g_2^*}\left( \frac{Wx}{\mu} \right)
		= \argmin_{p \in \R^n}{\left\{ \frac{1}{\mu} g_2^*(p) +\frac{1}{2}\left\| \frac{Wx}{\mu} - p \right\|^2  \right\}}
		&= \argmin_{p \in \left[ -\lambda, \lambda\right]^n}{\frac{1}{2}\left\| \frac{Wx}{\mu} - p \right\|^2} \\
		&= \mathcal{P}_{\left[ -\lambda, \lambda\right]^n}\left( \frac{Wx}{\mu} \right).
\end{align*}
Hence, choosing $\mu_k = \frac{1}{ak}$, for some parameter $a \in \R_{++}$ and taking into account that $L_k = \frac{\left\| A \right\|^2 + \|W\|^2}{\mu_k} = 2ak$, for $k \geq 1$, the iterative scheme \eqref{diff-scheme-variable-smoothing} with starting point $b\in\R^n$ becomes
\begin{empheq}[box=\fbox]{align*}
\text{Initialization} : & \ t_1=1, \ y_1=x_0=b \in \R^n,\ a>0,   \\
\text{For }	k\geq1 : & \ \mu_k=\frac{1}{ak}, \ L_k=2ak, \\
		& \ x_k = y_k - \frac{1}{L_k} \left( A\mathcal{P}_{\left[ -1, 1\right]^n}\left( \frac{Ay_k-b}{\mu_k} \right) + W\mathcal{P}_{\left[ -\lambda, \lambda\right]^n}\left( \frac{Wy_k}{\mu_k} \right) \right), \\
		& \ t_{k+1} = \frac{1+\sqrt{1+4t_k^2}}{2}, \\
		&  \ y_{k+1} = x_k + \frac{t_k-1}{t_{k+1}}(x_k - x_{k-1})
\end{empheq}
We considered the $256 \times 256$ cameraman test image, which is part of the image processing toolbox in Matlab, that we vectorized (to a vector of dimension $n=256^2 =65536$) and normalized, in order to make pixels range in the closed interval from $0$ (pure black) to $1$ (pure white). In addition, we added normally distributed white Gaussian noise with standard deviation $10^{-3}$ and set the regularization parameter to $\lambda=2\text{e-}5$. The original and observed images are shown in Figure \ref{fig:cameraman-original}.
\begin{figure}[ht]
	\centering
	\includegraphics*[viewport= 91 305 540 520, width=0.8\textwidth]{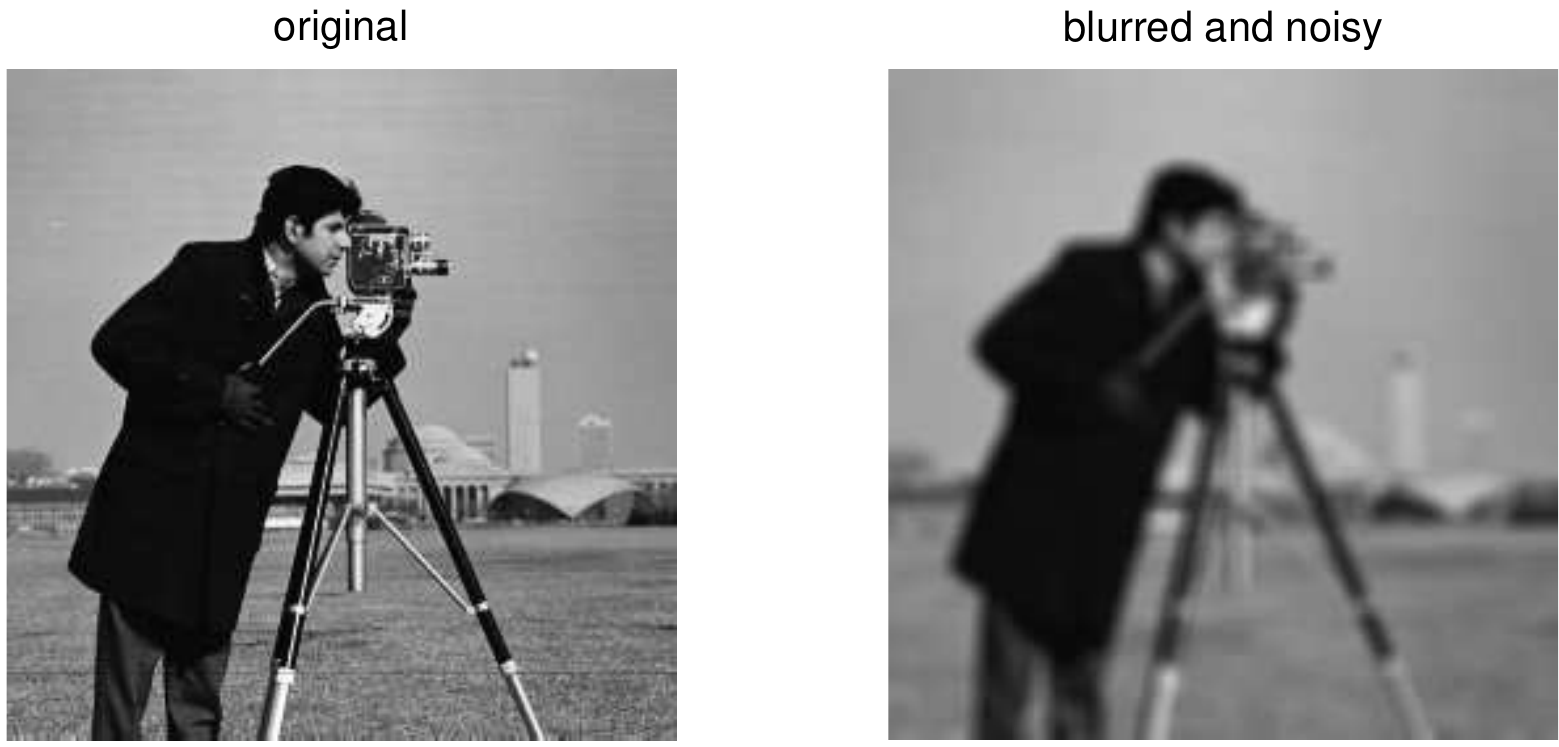}
	\caption{The $256 \times 256$ cameraman test image}
	\label{fig:cameraman-original}
\end{figure}
When measuring the quality of the restored images, we made use of the \textit{improvement in signal-to-noise ratio (ISNR)}, which is defined as
$$ \text{ISNR}_k = 10 \log_{10} \left( \frac{\left\| x - b \right\|^2}{\left\| x - x_k \right\|^2} \right), $$
where $x$, $b$ and $x_k$ denote the original, the observed and the estimated image at iteration $k \geq 1$, respectively. We tested several values for $a \in \R_{++}$ and we obtained after $100$ iterations the
objective values and the ISNR values presented in Table \ref{table:epsilon}.
\begin{table}[ht]
	\centering
	%\begin{tabular}{ l || c | c | c | c | c | c | c | c }
		%$\varepsilon$ & $1$e-$2$ & $1$e-$1$ & 1 & $1$e+$1$ & $1$e+$2$ & $1$e+$3$ & $1$e+$4$ & $1$e+$5$ \\ \hline \hline
		%fval & 684.379 & 676.106 & 612.818 & 346.427 & 113.142 & 55.521 & 53.850 & 54.801  \\
		%ISNR & 0.001 & 0.008 & 0.078 & 0.585 & 2.614 & 5.250 & 5.346 & 5.367
	%\end{tabular}
	\begin{tabular}{ l || c | c | c | c | c | c | c | c }
		$a$  & $1$e-$4$  & $1$e-$3$ & $1$e-$2$ & $1$e-$1$ &   $1$    & $1$e+$1$ & $1$e+$2$  & $1$e+$3$  \\ \hline \hline
		fval & $164.621$ & $80.915$ & $55.763$ & $53.669$ & $53.579$ & $63.754$ & $208.413$ & $531.022$   \\
		ISNR & $1.282$   & $3.839$  & $5.241$  & $5.352$  & $5.337$  & $4.351$  & $1.180$   & $0.199$
	\end{tabular}
	\caption{Objective values (fval) and ISNR values (higher is better) after $100$ iterations.}
	\label{table:epsilon}
\end{table}
 In the context of solving the problem \eqref{probimageproc} we compared the variable smoothing approach (VS) for $a=1$e-$1$ with the operator-splitting algorithm based on skew splitting (SS) proposed in \cite{BriCom11,ComPes12} with parameters $\varepsilon=\frac{1}{2(\sqrt{2} +1)}$ and $\gamma_k = \gamma = \frac{\varepsilon}{2} + \frac{1-\varepsilon}{2\sqrt{2}}$, for any $k \geq 1$, and with the primal-dual algorithm (PD) from \cite{ChaPoc11} with parameters $\theta=1$, $\sigma=0.01$ and $\tau=49.999$. The parameters considered for the three approaches provide the best results when solving \eqref{probimageproc}.
\begin{figure}[ht]	
	\centering
	\includegraphics*[viewport= 143 249 470 606, width=0.32\textwidth]{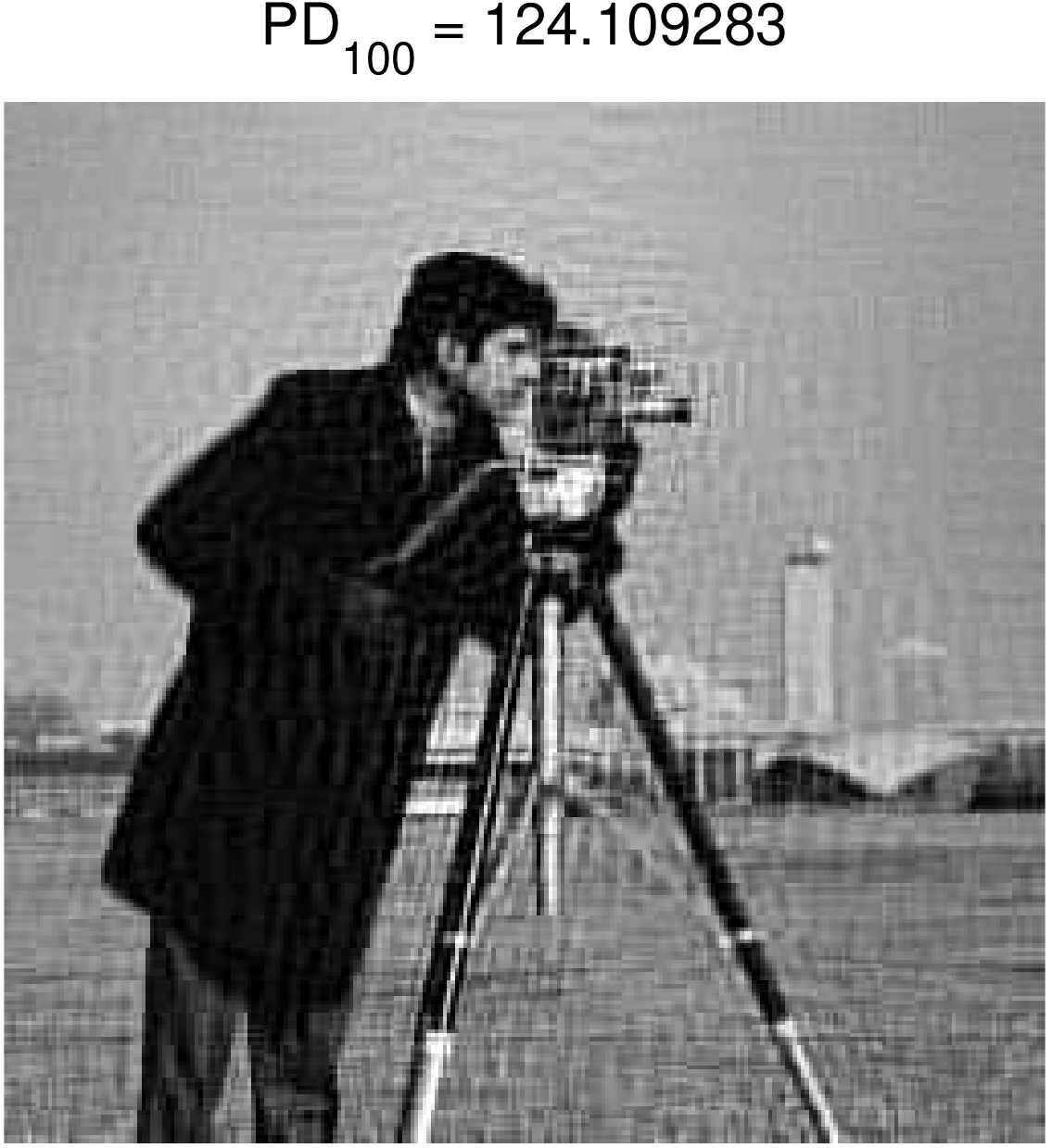}
	\includegraphics*[viewport= 143 249 470 606, width=0.32\textwidth]{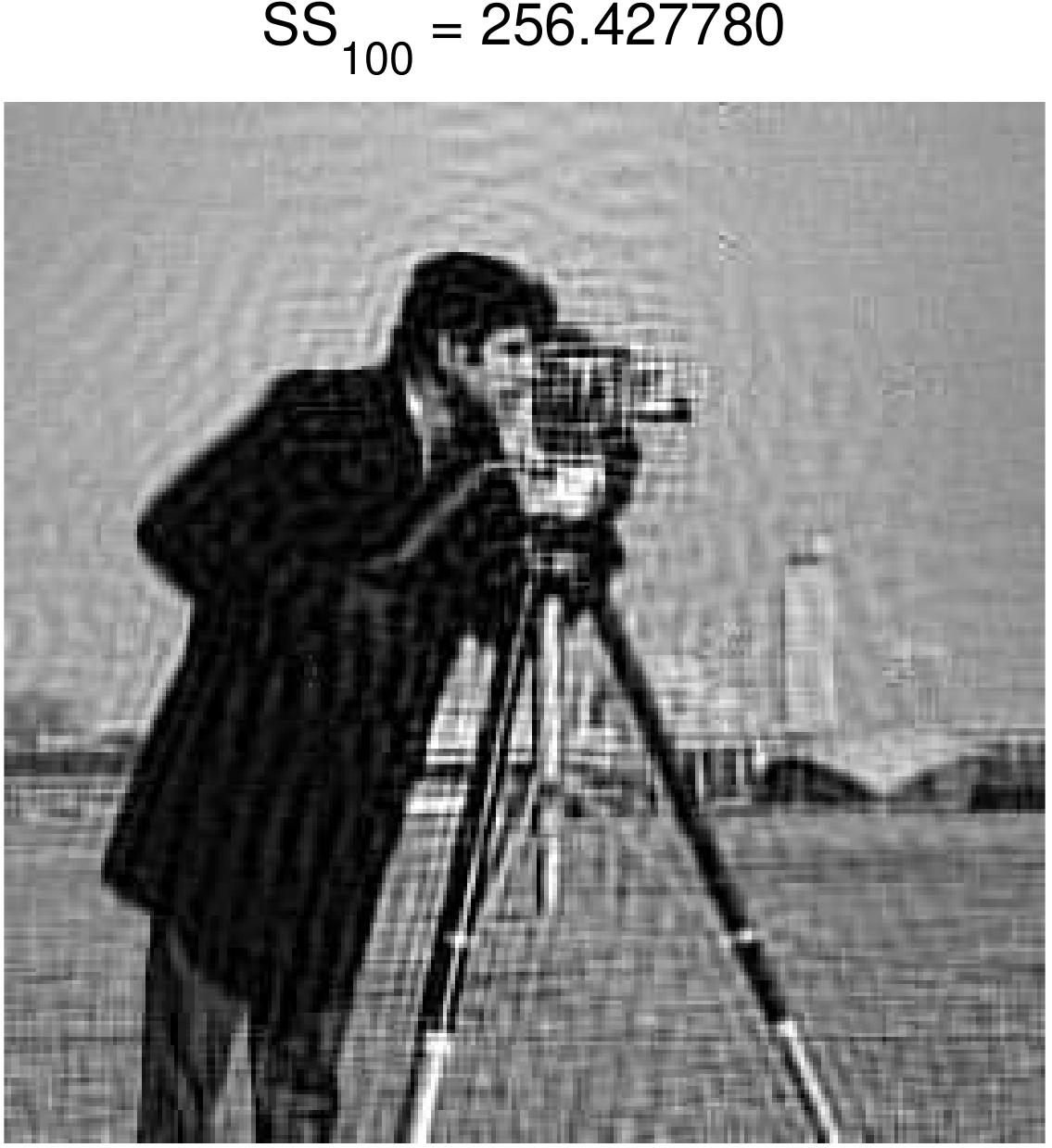}
	\includegraphics*[viewport= 143 249 470 606, width=0.32\textwidth]{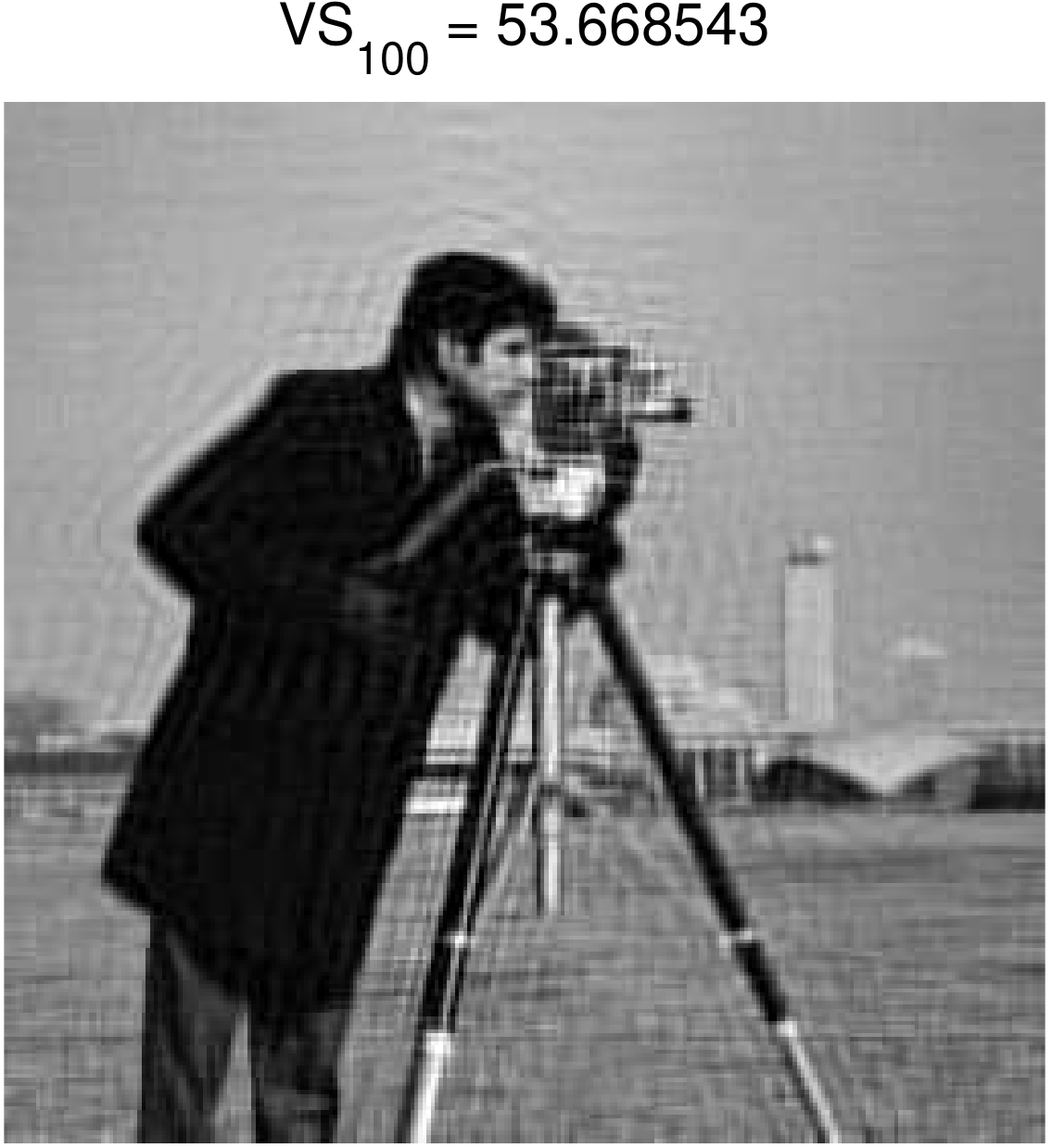}	
	\caption{Results furnished by the primal-dual (PD), the skew splitting (SS) and the variable smoothing (VS) algorithms after 100 iterations.}
	\label{fig:cameramen-PS}	
\end{figure}
The output of these three algorithms after $100$ iterations, along with the corresponding objective values, can be seen in Figure \ref{fig:cameramen-PS} and they show that the variable smoothing approach outperforms the other two methods.  Figure \ref{fig:cameramen-fval-ISNR} shows the evolution of the values of the objective function and of the improvement in signal-to-noise ratio within the first $100$ iterations.
\begin{figure}[ht]	
	\centering
	\includegraphics*[viewport= 57 247 552 597, width=0.49\textwidth]{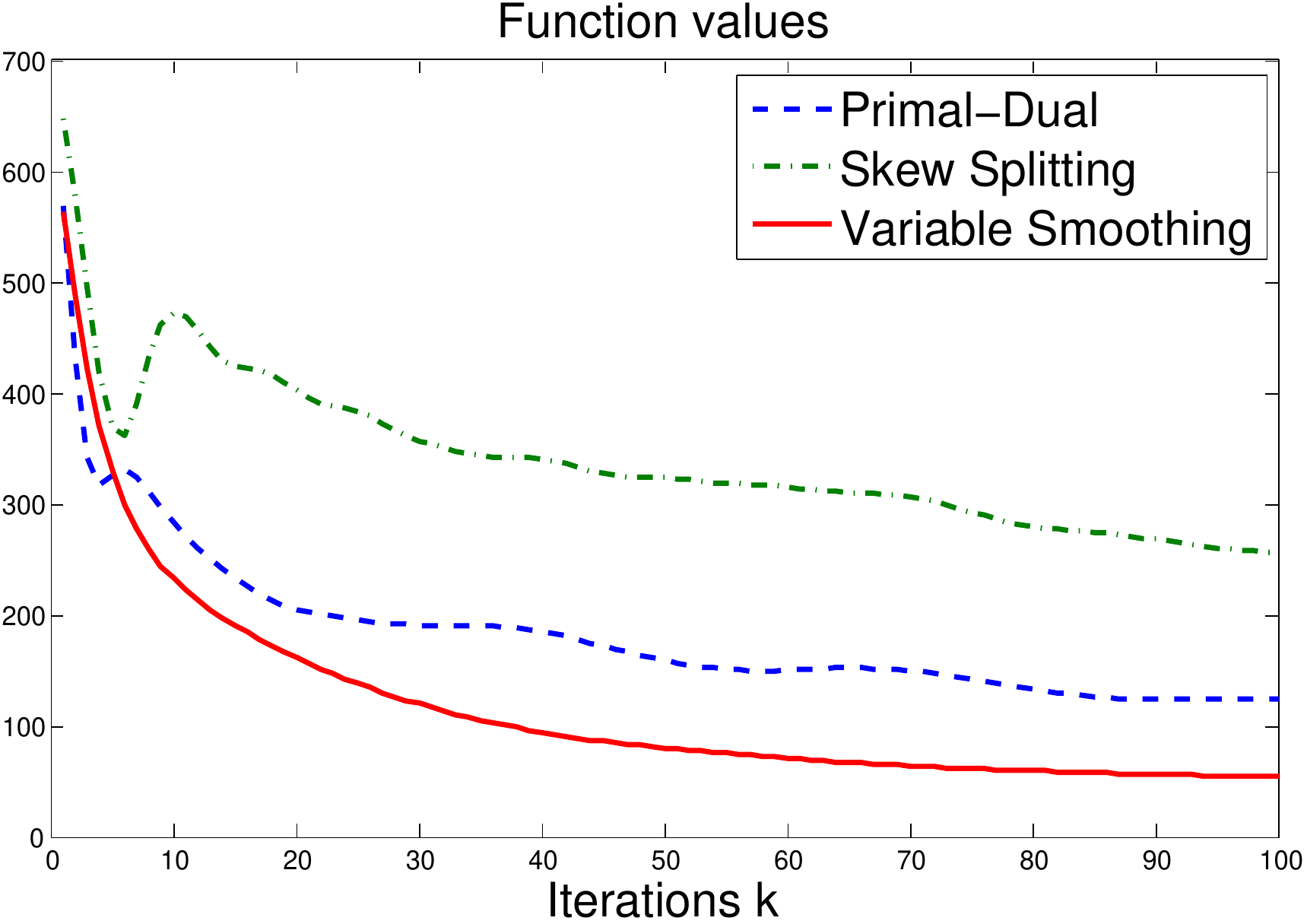}
	\includegraphics*[viewport= 57 247 552 597, width=0.49\textwidth]{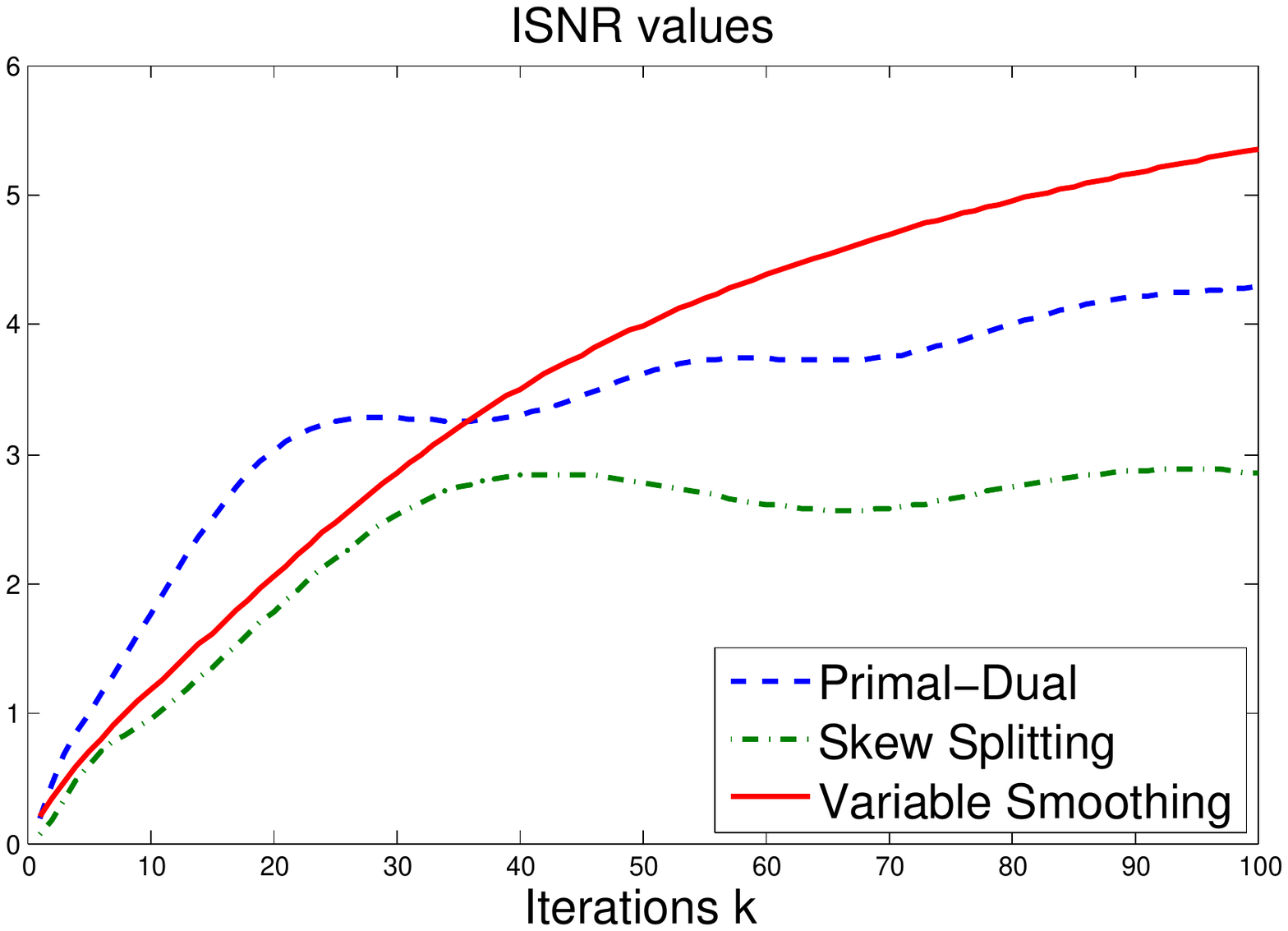}
	\caption{The evolution of the values of the objective function and of the ISNR for the primal-dual (PD), the skew splitting (SS) and the variable smoothing (VS) algorithms after 100 iterations.}
	\label{fig:cameramen-fval-ISNR}	
\end{figure}

\subsection{Support vector machines classification}\label{subsvmc}

The second numerical experiment we consider for the variable smoothing algorithm concerns the solving of the problem of classifying images via  support vector machines classification, an approach which belong to the class of kernel based learning methods.

The given data set consisting of $5268$ images of size $200 \times 50$ was taken from a real-world problem a supplier of the automotive industry was faced with by establishing a computer-aided quality control for manufactured devices at the end of the manufacturing process (see \cite{BotHeiWan12} for more details on this data set). The overall task is to classify fine and defective components which are labeled by $+1$ and $-1$, respectively.

The classifier functional $\verb"f"$ is assumed to be an element of the \textit{Reproducing Kernel Hilbert Space (RHKS)} $\h_\kappa$, which in our case is induced by the symmetric and finitely positive definite Gaussian kernel function
$$\kappa: \R^d \times \R^d \rightarrow \R, \  \kappa(x,y) = \exp \left( -\frac{\left\|x-y \right\|^2}{2 \sigma^2} \right).$$
Let $\langle \cdot,\cdot \rangle_\kappa$ denote the inner product on $\h_\kappa$, $\|\cdot\|_\kappa$  the corresponding norm and
$K\in\R^{n\times n}$ the \textit{Gram matrix} with respect to the training data set $\Z=\{(X_1,Y_1),\ldots,(X_n,Y_n)\}\subseteq \R^d \times \{+1,-1\}$, namely the symmetric and positive definite matrix with entries $K_{ij}=\kappa(X_i,X_j)$ for $i,j=1,\ldots,n$. Within this example we make use of the \textit{hinge loss} $v:\R \times \R \rightarrow \R$, $v(x,y)=\max \{1-xy,0\}$, which penalizes the deviation between the predicted value $\verb"f"(x)$ and the true value $y \in \{+1,-1\}$. The smoothness of the decision function $\verb"f" \in \h_\kappa$ is employed by means of the \textit{smoothness functional} $\Omega : \h_\kappa \rightarrow \R$, $\Omega(f)=\left\|\verb"f"\right\|_\kappa^2$, taking high values for non-smooth functions and low values for smooth ones. The decision function $\verb"f"$ we are looking for is the optimal solution of the \textit{Tikhonov regularization problem}
\begin{equation}
		\label{opt-problem:SVM-original}
		\inf_{\verb"f" \in \h_\kappa}{\left\{\frac{1}{2} \Omega (\verb"f") +  C \sum_{i=1}^n v(\verb"f"(X_i),Y_i) \right\}},
\end{equation}
where $C>0$ denotes the regularization parameter controlling the tradeoff between the loss function and the smoothness functional.

The \textit{representer theorem} (cf. \cite{ShaChr04}) ensures the existence of a vector of coefficients $c=(c_1,\ldots,c_n)^T \in \R^n$ such that the minimizer $\verb"f"$ of \eqref{opt-problem:SVM-original} can be expressed as a kernel expansion in terms of the training data, i.e., $\verb"f"(\cdot) = \sum_{i=1}^n c_i \kappa(\cdot, X_i)$. Thus, the smoothness functional becomes $\Omega(\verb"f")=\left\|\verb"f"\right\|_\kappa^2 = \<\verb"f",\verb"f"\>_\kappa = \sum_{i=1}^n \sum_{j=1}^n {c_i c_j \kappa(X_i,X_j)} = c^T Kc$ and for $i=1,\ldots,n$, it holds $\verb"f"(X_i) = \sum_{j=1}^n c_j \kappa(X_i,X_j) = (Kc)_i$. Hence, in order to determine the decision function one has to solve the convex optimization problem
\begin{align}\label{optsvmc}
\inf_{c \in \R^n}{\left\{f(c) +  C \sum_{i=1}^n g_i(Kc) \right\}},
\end{align}
where $f: \R^n \rightarrow \R$, $f(c) = \frac{1}{2}c^TKc$, and $g_i:\R^n \rightarrow \R$, $g_i(c)=Cv(c_i,Y_i)$ for $i=1,\ldots,n$. The function $f: \R^n \rightarrow \R$ is convex and differentiable and it fulfills $\nabla f(c) = Kc$ for every $c \in \R^n$, thus $\nabla f$ is Lipschitz continuous with Lipschitz constant $L_{\nabla f} = \|K\|$. For any $i=1,...,n$ the function $g_i : \R^n \rightarrow \R$ is convex and $C$-Lipschitz continuous, properties which allowed us to solve the problem \eqref{optsvmc} with algorithm \eqref{diff-scheme-variable-smoothing}, by using also the considerations made in Subsection \ref{submore}. For any $i=1,...,n$ and every $p=(p_1,...,p_n)^T \in \R^n$ it holds (see, also, \cite{BotLorenz, BotHeiWan12})
\begin{align*}
		g_i^*(p) & =  \sup_{c\in\R^n}{\left\{ \< p,c\> -Cv(c_i,Y_i) \right\}} = C\sup_{c\in\R^n}{\left\{ \< \frac{p}{C},c \> -v(c_i,Y_i) \right\}} \\
		&= \left\{\begin{aligned} &C(v(\cdot,Y_i))^*\left(\frac{p_i}{C}\right), \ \text{if }  p_j=0,\ i\neq j, \\ &+\infty, \ \text{otherwise,} \end{aligned} \right.\\
&= \left\{ \begin{aligned} & p_iY_i, \ \text{if } p_j=0,\ i\neq j \ \mbox{and} \ p_iY_i \in [-C,0], \\ &+\infty, \ \text{otherwise}. \end{aligned} \right.
\end{align*}
Thus, for $\mu \in \R_{++}$, $c=(c_1,...,c_n)^T$ and $i=1,...,n$ we have
\begin{align*}
	\Prox_{\frac{1}{\mu}g_i^*}\left( \frac{c}{\mu} \right)
	&= \argmin_{p \in \R^n}{\left\{ \frac{1}{\mu}g_i^*(p) + \frac{1}{2} \left\| \frac{c}{\mu} - p \right\|^2  \right\}} \\
	&= \argmin_{\substack{p_iY_i \in \left[-C,0\right] \\ p_j=0, j\neq i}}{\left\{\frac{p_iY_i}{\mu} + \frac{1}{2} \left(\frac{c_i}{\mu} - p_i \right)^2  \right\}} \\
	&= \argmin_{\substack{p_iY_i \in \left[-C,0\right] \\ p_j=0, j\neq i}}{\left\{p_iY_i + \frac{\mu}{2} \left(\frac{c_i}{\mu} - p_i \right)^2  \right\}}.
\end{align*}
For $Y_i=1$ we have
\begin{align*}
	\Prox_{\frac{1}{\mu}g_i^*}\left( \frac{c}{\mu} \right)
	= \argmin_{\substack{p_iY_i \in \left[-C,0\right] \\ p_j=0, j\neq i}}{\left\{ p_i + \frac{\mu}{2} \left( \frac{c_i}{\mu} - p_i \right)^2  \right\}}
	= \left( 0,\ldots, \mathcal{P}_{\left[-C,0\right]}\left( \frac{c_i - 1}{\mu} \right),\ldots, 0 \right)^T,
\end{align*}
while for $Y_i=-1$, it holds
\begin{align*}
	\Prox_{\frac{1}{\mu}g_i^*}\left( \frac{c}{\mu} \right)
	= \argmin_{\substack{p_iY_i \in \left[-C,0\right] \\ p_j=0, j\neq i}}{\left\{ -p_i + \frac{\mu}{2} \left( \frac{c_i}{\mu} - p_i \right)^2  \right\}}
	= \left( 0,\ldots, \mathcal{P}_{\left[0,C\right]}\left( \frac{c_i + 1}{\mu} \right),\ldots, 0 \right)^T.
\end{align*}
Summarizing, it follows
\begin{align*}
	\Prox_{\frac{1}{\mu}g_i^*}\left(\frac{c}{\mu}\right) = \left( 0,\ldots, \mathcal{P}_{Y_i[-C,0]}\left(\frac{c_i -Y_i}{\mu} \right),\ldots, 0 \right)^T.
\end{align*}
Thus, for every $c=(c_1,...,c_n)^T$ we have
\begin{align*}
	\nabla \left(\sum_{i=1}^n (^{\mu}g_i \circ K)\right)(c) & = \sum_{i=1}^n{\nabla (^{\mu}g_i \circ K)(c)} = \sum_{i=1}^n{K \Prox_{\frac{1}{\mu}g_i^*}\left( \frac{Kc}{\mu} \right)} \\
	&= K\left(\mathcal{P}_{Y_1\left[-C,0\right]}\left( \frac{(Kc)_1 -Y_1}{\mu} \right),..., \mathcal{P}_{Y_n\left[-C,0\right]}\left( \frac{(Kc)_n -Y_n}{\mu} \right) \right)^T.
\end{align*}
Using the nonexpansiveness of the projection operator, we obtain for every $c,d \in \R^n$
\begin{align*}
	\left\| \nabla \left(\sum_{i=1}^n (g_i^{\mu} \circ K)\right)(c) - \nabla \left(\sum_{i=1}^n (g_i^{\mu} \circ K)\right)(d)  \right\|
	\leq \left\| K \right\| \left\| \frac{Kc-Kd}{\mu} \right\|
	\leq \frac{\left\| K \right\|^2}{\mu} \left\| c-d \right\|.
\end{align*}
Choosing $\mu_k = \frac{1}{ak}$, for some parameter $a \in \R_{++}$ and taking into account that $L_k = \|K\| + ak \left\| K \right\|^2$, for $k \geq 1$, the iterative scheme \eqref{diff-scheme-variable-smoothing} with starting point $x_0 = 0 \in\R^n$ becomes
\begin{empheq}[box=\fbox]{align*}
\text{Initialization} : & \ t_1=1, \ y_1=x_0=0 \in \R^n, \ a\in\R_{++},   \\
\text{For } k \geq 1 : & \ \mu_k=\frac{1}{ak},\ L_k=\left\| K \right\| + ak \left\| K \right\|^2, \\
		& \ x_k = y_k - \frac{1}{L_k} \left(Ky_k +  K\left(\mathcal{P}_{Y_i\left[-C,0\right]}\left( \frac{(Kc)_i -Y_i}{\mu} \right)\right)_{i=\overline{1,n}}^T \right),\\
		& \ t_{k+1} = \frac{1+\sqrt{1+4t_k^2}}{2}, \\
		& \ y_{k+1} = x_k + \frac{t_k-1}{t_{k+1}}(x_k - x_{k-1})
\end{empheq}

\begin{figure}[ht]	
	\centering
	\includegraphics*[viewport= 277 355 326 506, width=0.08\textwidth]{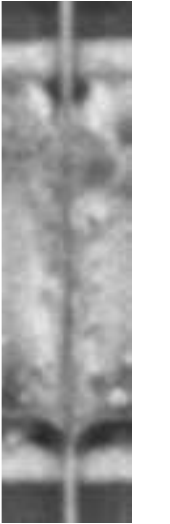}
	\includegraphics*[viewport= 277 355 326 506, width=0.08\textwidth]{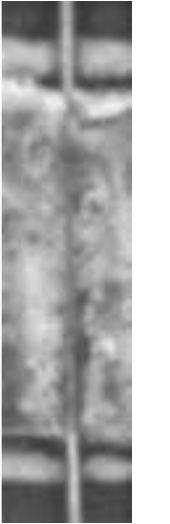}
	\includegraphics*[viewport= 277 355 326 506, width=0.08\textwidth]{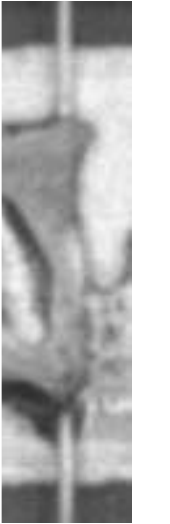}	
	\includegraphics*[viewport= 277 355 326 506, width=0.08\textwidth]{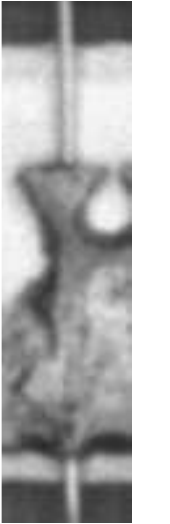}	
	\caption{Example of two fine and two defective devices.}
	\label{fig:classification-good-bad}	
\end{figure}
Coming to the real-data set, we denote by $\D=\left\{(X_i,Y_i), i=1,\ldots,5268 \right\} \subseteq \R^{10000} \times \{+1,-1\}$ the set of all data available consisting of $2682$ images of class $+1$ and $2586$ images of class $-1$. Notice that two examples of each class are shown in Figure \ref{fig:classification-good-bad}. Due to numerical reasons, the images have been normalized (cf. \cite{LalChaSch06}) by dividing each of them by the quantity $\left(\frac{1}{5268} \sum_{i=1}^{5268}{\left\|X_i \right\|^2}\right)^{\frac{1}{2}}$.
\begin{table}[ht]
	\centering
		\begin{tabular}{ l || c | c | c | c | c | c | c | c | c }
		$a$   & $1$e-$5$ & $1$e-$4$ & $1$e-$3$ & $1$e-$2$ & $1$e-$1$ & $1$ & $1$e+$1$ & $1$e+$2$ & $$1e+$3$ \\ \hline \hline
		err & $0.4176$ & $0.3037$ & $0.2278$ & $0.2468$ & $0.3986$ & $0.5315$ & $0.5125$ & $1.5945$ & $48.9561$
	\end{tabular}
	\caption{Average classification errors in percentage.}
	\label{table:classification-epsilon}
\end{table}
We considered as regularization parameter $C=100$ and as kernel parameter $\sigma=0.5$, which are the optimal values reported in \cite{BotHeiWan12} for this data set from a given pool of parameter combinations, tested different values for $a \in \R_{++}$ and performed for each of those choices  a 10-fold cross validation on $\D$. We terminated the algorithm after a fixed number of $10000$ iterations was reached, the average classification errors being presented in Table \ref{table:classification-epsilon}. For $a=1$e-$3$ we obtained the lowest missclassification rate of $0.2278$ percentage. In other words, from $527$ images
belonging to the test data set an average of $1.2$ were not correctly classified.

\end{document}